\documentclass{birkjour}
\usepackage{amssymb}

 \newtheorem{thm}{Theorem}[section]
 \newtheorem{cor}[thm]{Corollary}
 \newtheorem{lem}[thm]{Lemma}
 
 \theoremstyle{definition}
 \newtheorem{defn}[thm]{Definition}
 \theoremstyle{remark}
 \newtheorem{rema}[thm]{Remark}
 
 \theoremstyle{example}
 \newtheorem{exmp}[thm]{Example}
 \numberwithin{equation}{section}

\def\rem#1#2{ #1\;\textnormal{rem}\;#2}

\def\res#1#2{ eval\left(#1;\  #2\right)}

\begin{document}

%-------------------------------------------------------------------------
% editorial commands: to be inserted by the editorial office
%
%\firstpage{1} \volume{228} \Copyrightyear{2004} \DOI{003-0001}
%
%
%\seriesextra{Just an add-on}
%\seriesextraline{This is the Concrete Title of this Book\br H.E. R and S.T.C. W, Eds.}
%
% for journals:
%
%\firstpage{1}
%\issuenumber{1}
%\Volumeandyear{1 (2004)}
%\Copyrightyear{2004}
%\DOI{003-xxxx-y}
%\Signet
%\co-mmby{inhouse}
%\submitted{March 14, 2003}
%\received{March 16, 2000}
%\revised{June 1, 2000}
%\accepted{July 22, 2000}
%
%
%
%---------------------------------------------------------------------------
%Insert here the title, affiliations and abstract:
%

\title[]
{\centering{Some Partial Fraction Identities associated \\ with the Cyclotomic Polynomials}}

%----------Author 1
\author[]{\centering{N. Uday Kiran}}

\address{
Department of Mathematics and Computer Science,\\
Sri Sathya Sai Institute of Higher Learning,\\
Prashanthi Nilayam, Puttaparthi,\\
Andhra Pradesh 515134, India.
}
\email{nudaykiran@sssihl.edu.in}
%----------classification, keywords, date
%\date{May 31, 2012}
%----------additions
%\dedicatory{Dedicated to Bhagawan Sri Satya Sai Baba}
%%% ----------------------------------------------------------------------

\begin{abstract}
We establish some partial fraction identities for rational functions whose denominators are implicit products of the cyclotomic polynomials. To achieve this, we first develop a general algebraic approach for partial fraction decomposition inspired by the Heaviside's cover-up method. We thus call our method the Extended Cover-Up Method. Using our method we obtain direct formulas for $q$-partial fractions for certain generating functions. As a direct consequence of our formulas one can compute the Sylvester denumerants, the Frobenius number and the Ehrhart polynomials in pseudo-polynomial time. Further, we provide a framework for a generalization of the Fourier-Dedekind sum and their associated Rademacher reciprocity theorem extending the results of Carlitz, Zagier and Gessel. By performing a Fourier analysis we demonstrate that our extended cover-up method explains in simple terms the mechanism behind the reciprocity law.
\end{abstract}

%%% ----------------------------------------------------------------------
\maketitle
%%% ----------------------------------------------------------------------
%\tableofcontents

\section{Introduction}
More than a century ago, Oliver Heaviside suggested an intriguing approach to perform partial fractions of rational functions with denominators having only linear factors. Given distinct numbers $a_{1},\cdots, a_{k}$ and $h(x)$ a polynomial of degree less than $k$, suppose the expression for the partial fraction is 
\begin{equation}\label{heaviside}
h(x)\prod_{j=1}^{k}\frac{1}{x-a_{j}}=\sum_{j=1}^{k}\frac{A_{j}}{x-a_{j}},
\end{equation}
then by the Heaviside's cover-up method we have  
$$
A_{i}=h(a_{i})\prod^{k}_{\substack{j=1 \\ j\neq i}}\frac{1}{a_{i}-a_{j}},
$$
that  is, to obtain $A_{i}$ we evaluate  the left hand side of  (\ref{heaviside}) at $x=a_{i}$ `covering' the $(x-a_{i})$ factor. In this work, we develop an algebraic approach to rigorously extend this method to higher degree polynomials. In order to achieve this, we define a `polynomial valued' $eval$ function (Definition \ref{eval}) that mimics the Heaviside's approach. Indeed, given $p_{1}(x), \cdots, p_{k}(x)$ pairwise relatively prime polynomials we obtain the partial fraction decomposition  
\begin{equation}\label{ECM}
\prod_{j=1}^{k}\frac{1}{p_{j}(x)}=\sum_{j=1}^{k}\frac{1}{p_{j}(x)}\res{\frac{1}{p_{1}(x)\cdots\widehat{p_{j}(x)}\cdots p_{k}(x)}}{p_{j}(x)},
\end{equation}
where $\widehat{   }$ stands for dropping the corresponding term.

Employing the extended cover-up method and exploiting the properties of the cyclotomic polynomials we develop some new and elegant partial fraction identities. Furthermore, equipped with our method, we develop direct formulas for certain partition of numbers and generalize their corresponding Rademacher reciprocity theorems. 

The problem of \textit{partition of a number into finite parts} is concerned with determining the number of ways a given positive number $t$ can be expressed as a linear combination of a given set of positive numbers $\{n_{1},\cdots,n_{k}\}$, i.e., the number of solutions of the equation 
\begin{equation}\label{linear_equation}
\alpha_{1}n_{1}+\cdots+\alpha_{k}n_{k}=t
\end{equation}
for any non-negative integers $\alpha_{1},\cdots,\alpha_{k}$. The number of such solutions is called the Sylvester denumerant and denoted $d(n_1,\cdots,n_k;t)$. It is well known in the literature that the formal power series expansion of the following rational function gives us the denumerants:
\begin{equation}\label{n_k}
\frac{1}{(1-x^{n_1})\dots (1-x^{n_k})}=\sum_{t=0}^{\infty}d(n_{1},\cdots,n_{k};t)x^{t}. 
\end{equation}
Once a partial fraction of (\ref{n_k}) is obtained one can compute the power series efficiently; thus obtain a direct formula for the denumerants. In this work, we assume $n_{1},\cdots,n_{k}$ are pairwise relatively prime. 

Munagi \cite{munagi2} has formulated a non-standard approach to partial fractions, called $q$-partial fractions, the denominators of which are always of the form $(1-x^{r})^{s}$ for some integers $r$ and $s$. Employing the techniques developed in this work, we perform a $q$-partial fraction of 
\begin{equation}\label{qpf_n_k}
F(x)=\frac{p(x)}{(1-x)^{m}}\prod_{j=1}^{k}\frac{1}{1-x^{n_j}}=\frac{g_{0}(x)}{(1-x)^{k+m}}+\sum_{j=1}^{k}\frac{g_{j}(x)}{1-x^{n_j}},
\end{equation}
where $0\leq m, \textnormal{ deg }p(x)<(n_{1}+\cdots+n_{k}+m),\textnormal{ deg }g_{0}(x)<k+m$ and $\textnormal{deg }g_{j}(x)<n_{j}$ for $1\leq j$.  An advantage of this representation is in expressing the formal power series easily with just binomial expansions which will in turn make it convenient to compute the denumerants. 

Many works have addressed the problem of computing the denumerants and its allied computationally intractable problems (see \cite{Alfonsin}). It is well known that the problem of deciding $d(n_{1},\cdots, n_{k}; t)>0$ is NP-Complete and determining the largest $t$ such that $d(n_{1},\cdots, n_{k}; t)=0$ (the Frobenius Problem) is NP-Hard. For the Frobenius Problem a geometric based polynomial time algorithm is provided by Kannan \cite{Kannan} when $k$ is fixed. However, this algorithm takes super-exponential time in $k$. In \cite{Komatsu}, Komatsu provided a formula through trigonometric identities. The formula we derive in this work provides a pseudo-polynomial time algorithm for arbitrary $k$. 

A special feature of our decomposition (\ref{qpf_n_k}) is that the `polynomial part' $g_{0}(x)$ can be expressed  in terms of the Bernoulli numbers and the Stirling numbers of first kind, and the `periodic part' $g_{j}(x)$, for $j\geq 1$, is the inverse Fourier transform of generalizations of the Fourier-Dedekind sums. 

The Fourier-Dedekind sums are considered as the building blocks of number theory and they unify many variations of the Dedekind sums. These sums appear in various contexts besides analytic number theory and discrete geometry (see for instance \cite{Beck,Beck1,Zagier,Tuskerman} and references therein). The Fourier-Dedekind sum is (here we skip a $1/b$ factor for expository purpose)
$$
s_{t}(a_{1},\cdots,a_{k};b )=\sum^{b-1}_{i=1}\frac{\xi^{it}}{(1-\xi^{ia_{1}})\cdots(1-\xi^{ia_{k}})}
$$ 
for $\xi=e^{2\pi i/b}$. They satisfy a nice reciprocity relation 
\begin{equation}\label{FDS_intro}
\textnormal{poly}(-t)=\sum^{k}_{j=1}\frac{1}{n_{j}}s_{t}(n_{1},\cdots,\widehat{n_{j}},\cdots,n_{k};\ n_{j}),
\end{equation}
where again $\widehat{}$ refers to dropping the corresponding term. 

It is well known that in deriving the reciprocity theorems of the Fourier-Dedekind sums the partial fraction decomposition plays a powerful role \cite{Beck, Gessel}. Moreover, one can see a striking resemblance between the right hand sides of (\ref{ECM}) and that of (\ref{FDS_intro}) (compare the occurrences of $p_{j}(x)$ and $n_{j}$). We believe that, through the lens of the extended cover-up, every partial fraction associated with the cyclotomic polynomials is a `potential' reciprocity theorem. A schematic diagram of our framework is given in the figure.
\begin{figure}[h]
\centering 
\includegraphics[scale=0.5]{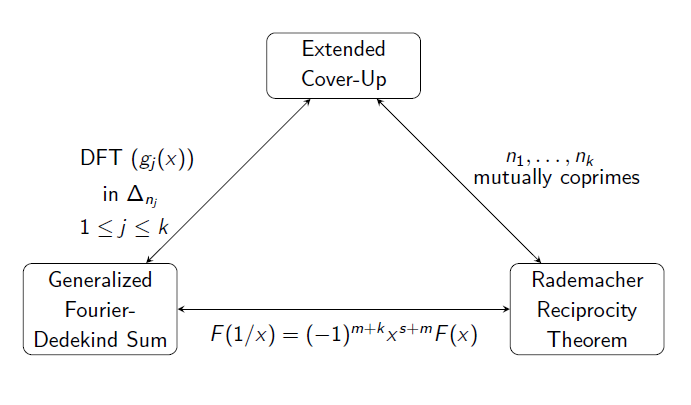}
%\caption{Scheme for Reciprocity Theorem}
\end{figure}
 
In view of the Fourier analysis of the $eval$ function, we generalize the Fourier-Dedekind sums by mainly introducing a polynomial term in the numerator. We prove the associated reciprocity theorems using the extended cover-up method. These reciprocity theorems for the generalized Fourier-Dedekind sums extend the results of Carlitz \cite{Carlitz}, Zagier \cite{Zagier} and Gessel \cite{Gessel}. Further, as an application of the reciprocity theorem of the $0$-dimensional generalized Fourier-Dedekind sum we derive formula for a trigonometric sum that plays a crucial role in the computation of our $q$-partial fractions. 

The results of this paper and further generalizations of the Fourier-Dedekind sum also hold for the Finite Fields (or the Galois Fields) \cite{UK_LS}.  

The paper is organized as follows. After briefly discussing the notation and terminology in Section \ref{sec_cyclotomic} and Section \ref{BT} we define our $eval$ symbolic evaluation function and its Fourier Transform.  In Subsection \ref{subsec_ECU}, we first prove a Fundamental Lemma and as a direct consequence we derive the extended cover-up method in Theorem \ref{cover_up}. In Section \ref{CP}, making use of the Cauchy-Euler derivative operator we recast a factorization result of cyclotomic polynomials into Bezout's identity thereby proving our main result Theorem \ref{main1} and provide some special cases. In Subsections \ref{di_lemma} we discuss some preparation lemmas and in Subsection \ref{subsec_qpf}, we prove the main $q$-partial fraction result. In Section \ref{qPFD}, we provide direct formula for denumerants and Ehrhart polynomials. In Section \ref{sec_recip}, we give a generalization of the Fourier-Dedekind sum and prove the associated reciprocity theorems. In Subsection \ref{0_dim}, we discuss the special case of $0$-dimensional generalized Fourier-Dedekind sum. In Subsection \ref{computation} we derive a formula for $f_{k}^{(m)}(1)$.

\section{Cyclotomic Polynomials}\label{sec_cyclotomic}
We recall the $n^{th}$ cyclotomic polynomial in $\mathbb{Q}[x]$, denoted $\Phi_{n}(x)$, is a polynomial of degree $\phi(n)$ whose roots are the $n^{th}$-primitive root of unity. By the set of $n^{th}$ primitive roots we mean 
$$
\Delta_{n}=\{e^{2\pi i r/n}:\  \textnormal{gcd}(r,n)=1,\ 1\leq r\leq n\}.
$$ 
Cyclotomic polynomials are irreducible polynomials with integer coefficients satisfying 
$$
1-x^{m}=\prod_{d|m} \Phi_{d}(x)
$$ 
and $\Phi_{1}(x)=1-x$. Note that some authors consider $(x-1)$ for $\Phi_{1}(x)$. 

We denote a special product of cyclotomic polynomials by
$$
\Psi_{n}(x)=\prod_{\substack{d|n\\ d\neq 1}}\Phi_{d}(x)\quad \textnormal{so that } 1-x^{n}=\Phi_{1}(x)\Psi_{n}(x).
$$ 
for $n\geq 2$. Therefore, we have 
$$
\Psi_{n}(x)=\frac{1-x^{n}}{1-x}=x^{n-1}+x^{n-2}\cdots+x+1.
$$
If $n=p$ is a prime number we have $\Psi_{p}(x)=\Phi_{p}(x)$. The notation $\Psi_{n}(x)$ was used for the purpose of defining inverse cyclotomic polynomials in \cite{Moree}. In this work, we denote the inverse cyclotomic polynomial defined in \cite{Moree} by
$$
\Theta_{n}(x)=\prod_{\substack{d|n\\ d\neq n}}\Phi_{d}(x).
$$

In order to derive the partial fractions, we exploit two key features of these polynomials. Firstly, the `cancellation of coefficients', when certain cyclotomic polynomials are multiplied they produce sparse polynomials. Secondly, efficient computation of the remainders with $\Phi_{d}(x)$ and $\Psi_{d}(x)$ (Lemma \ref{cor_phi}) and the computation of the Fourier Transform of the $eval$ function in the cyclotomic fields. 

By the formula for partial fraction decomposition given in Theorem \ref{main1} we can develop a formula for a general decomposition such as 
\begin{equation}\label{factorization}
\prod_{j=1}^{k}\frac{1}{1-x^{n_{j}}} = \frac{1}{(1-x)^{k}}\prod_{j=1}^{k}\prod_{\substack{d_{j}|n_{j}\\ d_{j}\neq 1}}\frac{1}{\Phi_{d_{j}}(x)}
\end{equation}
when $n_{1},\cdots, n_{k}$ are pairwise relatively prime. Nevertheless, after developing the general case we shift our focus to the $q$-partial fractions as given in (\ref{qpf_n_k}) owing to their direct relevance to number theoretic problems.

\section{Our Algebraic Approach to Partial Fractions}\label{BT}
In this paper, we restrict ourselves to those polynomials and rational functions with coefficients from $\mathbb{Q}$, the field of rationals. In particular, we concern ourselves with elements from the ring of polynomials $\mathbb{Q}[x]$ and the field of fractions $\mathbb{Q}(x)$. Our results can also be easily extended to the Finite Field case. We follow the notation and terminology as in \cite{Loren}.  

An element $p(x)\in \mathbb{Q}[x]$ is a polynomial of the form 
$$
p(x)=\alpha_{m}x^{m}+\cdots+\alpha_{0},\ \  \alpha_{i}\in \mathbb{Q}, \textnormal{ for } i=1,\cdots, m.
$$ 
If $\alpha_{m}\neq 0$ we say that the polynomial $p(x)$ is of degree $m$ and denoted $\textnormal{deg }p(x)=m$. The polynomial $p(x)$ is said to be irreducible if its only divisors are the polynomial itself and a constant polynomial in $\mathbb{Q}[x]$. We refer to $Dp(x)$ or $p'(x)$ as the (algebraic) derivative of $p(x)$.

It is well known that $\mathbb{Q}[x]$ is a principal ideal domain. We denote the ideal generated by a polynomial with the same letter in italic font, for instance the prime ideal generated by $p(x)$ is $\mathfrak{p}$, and denote the quotient of $\mathbb{Q}[x]$ with $\mathfrak{p}$ by $\mathbb{Q}[x]/\mathfrak{p}$. Given a non-empty multiplicative set $S$ (i.e., if $a,b\in S$ then $ab\in S$ and $0\notin S$) of $\mathbb{Q}[x]$, the ring of fractions of $\mathbb{Q}[x]$ with respect to $S$ will be denoted by $S^{-1}\mathbb{Q}[x]$ and is defined by 
$$
S^{-1}\mathbb{Q}[x]=\left\{\frac{f(x)}{g(x)}\in \mathbb{Q}(x):f(x)\in \mathbb{Q}[x] \textnormal{ and } g(x)\in S\right\}.
$$  
In this paper, we only work with \textit{proper fractions} of the form
$$
\frac{f(x)}{g(x)}\in \mathbb{Q}(x)\quad \textnormal{ with }\quad \textnormal{deg } f(x)< \textnormal{deg } g(x),
$$ 
and $g(x)$ having zeros on the unit circle, or in other words, $g(x)$ is a product of cyclotomic polynomials. 

\subsection{Bezout's Identity and Partial Fraction Decomposition}\label{PFD}

Two polynomials $g_{1}(x),g_{2}(x)\in \mathbb{Q}[x]$ are said to be relatively prime (or coprime) if their greatest common divisor (gcd) is 1. By Bezout's property there exist $a_{1}(x), a_{2}(x)\in \mathbb{Q}[x]$ such that 
$$
a_{1}(x)g_{1}(x)+a_{2}(x)g_{2}(x)=1.
$$
Similarly, if $g_{1}(x),\cdots,g_{n}(x)\in \mathbb{Q}[x]$ are relatively prime (or mutually coprime) then there exist $a_{1}(x),\cdots,a_{n}(x)\in \mathbb{Q}[x]$ such that 
$$
a_{1}(x)g_{1}(x)+\cdots+a_{n}(x)g_{n}(x)=1.
$$  
In this section, we prove two key results that provide an algebraic approach to partial fraction decomposition through the Bezout's identity. Towards this end, we define a natural operation we call the \textit{evaluation function modulo a polynomial} $\left(\textnormal{denoted }\res{\frac{r(x)}{s(x)}}{p(x)}\right)$, which evaluates symbolically the fraction $\frac{r(x)}{s(x)}$ in a localized ring modulo $p(x)$. The definition of our $eval$ function is algebraic and relies on the principles of substitution and localization.

\subsection{The $\textnormal{rem}$ Operator}

We define the remainder operator in the standard manner.

\begin{defn} 
Given $f(x),p(x)\in \mathbb{Q}[x]$, the remainder (denoted $\rem{f(x)}{p(x)}$) of $f(x)$ divided by $p(x)$ is a unique polynomial $r(x)$ given by the Euclid's algorithm  
$$
f(x)=p(x)q(x)+r(x),
$$
where $r(x)\in \mathbb{Q}[x]$ and $\textnormal{deg } r(x)<\textnormal{deg } p(x)$.
\end{defn}

For the purpose of defining the $eval$ operator we will express rem operator as a composition of two functions: the function $\pi_{\mathfrak{p}}:\mathbb{Q}\rightarrow \mathbb{Q}/\mathfrak{p}$ is defined 
$$
\pi_{\mathfrak{p}}(f(x))=f(x)+\mathfrak{p}
$$
and $\eta:\mathbb{Q}[x]/\mathfrak{p}\rightarrow \mathbb{Q}[x]$ defined
$$
\eta(A)=
\begin{cases} 
0 & \textnormal{ if } 0\in A\\
\textnormal{smallest degree polynomial in } A & \textnormal{ if } 0\notin A. 
\end{cases}
$$
We can easily prove the equality: $(\rem{f(x)}{p(x)})=\eta\circ \pi_{\mathfrak{p}}(f(x)).$  

A key strategy in determine the remainder of $f(x)$ with a polynomial $p(x)$, suggest by the above composition, is to replace all the occurrences of $p(x)$ in $f(x)$ by the zero polynomial until one obtains the smallest degree polynomial, which essentially is the polynomial of degree less than that of $p(x)$. That is, one can obtain the remainder by substitutions and not by performing long division (see for instance \cite{Laudano}). 

%\begin{exmp}
%The remainder of $f(x)=x^{4}+2x^{3}+x^{2}+5$ when divided by $p(x)=x^{2}+1$ can be easily obtained by substituting $-1$ wherever $x^{2}$ occurs in $f(x)$. Thus, we have 
%\begin{eqnarray}
%\nonumber \rem{(x^{4}+2x^{3}+x^{2}+5)}{(x^{2}+1)}&=& (-1)(-1)+2(-1)x+(-1)+5\\
%\nonumber &=& -2x+5 
%\end{eqnarray}
%\end{exmp}

\begin{lem}\label{cor_xk}
Suppose $f(x)=a_{n}x^{n}+a_{n-1}x^{n-1}+\cdots+a_{0}$ and if $n>k$ then
$$
\rem{f(x)}{x^{k}}=a_{k-1}x^{k-1}+a_{k-2}x^{k-2}+a_{0}.
$$
\end{lem}

Computationally, substitution based remainder algorithm is faster than long division. Incidentally, the cyclotomic polynomials have a good behavior with respect to taking remainders. The following corollary provides us with an efficient way of computing the remainder through the cyclotomic polynomial $\Phi_{d}(x)$ and $\Psi_{m}(x)$

\begin{lem}\label{cor_phi}
Suppose $f(x)=a_{n}x^{n}+a_{n-1}x^{n-1}+\cdots+a_{0}$ then
$$
\rem{f(x)}{\Phi_{d}(x)}=\rem{(a_{n}x^{n\%d}+a_{n-1}x^{(n-1)\%d}+a_{0})}{\Phi_{d}(x)},
$$
and 
$$
\rem{f(x)}{\Psi_{m}(x)}=\rem{(a_{n}x^{n\%m}+a_{n-1}x^{(n-1)\%m}+a_{0})}{\Psi_{m}(x)},
$$
where $\%$ operator is the remainder in the ring of integers. 
\end{lem}
\begin{proof}
As $\Phi_{d}(x)$ divides $1-x^{d}$, we first divide the polynomial $f(x)$ by $(1-x^{d})$ to obtain 
\begin{equation}\label{by_xd}
f(x)=(1-x^{d})q_{d}(x)+r_{d}(x),
\end{equation}
where the remainder $r_{d}(x)$, in turn when divided by $\Phi_{d}(x)$ gives us
$$
r_{d}(x)=\Phi_{d}(x)q(x)+r(x) \textnormal{ where deg }r(x)<\phi(d).
$$
On substituting for $r_d(x)$ into (\ref{by_xd}) and $(1-x^{d})=\Phi_{d}\prod_{d'|d, d'\neq d}\Phi_{d'}(x)$ we get 
$$
f(x)=\left(\prod_{\substack{d'|d\\ d'\neq d}}\Phi_{d'}(x)+q(x)\right)\Phi_{d}(x)+r(x), \textnormal{ where deg }r(x)<\textnormal{deg }\Phi_{d}(x).
$$
Moreover, note that the term $r_{d}(x)$ can be obtained by replacing $x^{d}$ by $1$ in $f(x)$. The result for $\Psi_{m}(x)$ follows in similar lines. 
\end{proof}

\begin{exmp}
The remainder of $x^{101}+5x^{31}$ by $\Phi_{3}(x)=x^{2}+x+1$ can be obtained by first observing that the polynomial $x^{2}+x+1$ divides $x^{3}-1$ so before substituting $x^{2}$ with $-x-1$ we can first substitute $x^{3}$ with $1$ leading us to 
\begin{eqnarray}
\nonumber \rem{(x^{101}+5x^{31})}{(x^{2}+x+1)}&=& \rem{(x^{101\%3}+5x^{31\%3})}{(x^{2}+x+1)}\\
\nonumber &=& \rem{(x^2+5x)}{(x^{2}+x+1)}\\
\nonumber &=& (-x-1)+5x=4x-1
\end{eqnarray}
\end{exmp}

In order to obtain the remainder by $\Psi_{m}(x)$ one requires regrouping the terms and performing a substitution based on the following rule:
\begin{equation}\label{rem_psi}
x^{i} \textnormal{ rem } \Psi_{m}(x)= \left\{\begin{array}{@{}l@{\thinspace}l}
       x^{i\%m} & \text{   if   } i\%m \neq m-1 \\
       -(x^{m-2}+\cdots+1)  & \text{   if   } i\%m=m-1. \\
     \end{array}\right.
\end{equation}

\subsection{A Symbolic Evaluation Function: $eval$}
Let $r(x),s(x)$ and $a(x)$ in the ring of polynomials $\mathbb{Q}[x]$ such that $s(x)$ and $a(x)$ are relatively prime. By Bezout's identity there exist two polynomials $\alpha(x),\beta(x)\in \mathbb{Q}[x]$ such that 
\begin{equation}\label{a_alpha}
\alpha(x)s(x)+\beta(x)a(x)=1.
\end{equation}
Let $a(x)$ factorize as 
\begin{equation}\label{ax}
a(x)=p_{1}(x)^{e_1}\cdots p_{k}(x)^{e_k},
\end{equation}
where the polynomials $p_{1}(x),\cdots, p_{k}(x)$ are distinct irreducible polynomials. We localize the ring $\mathbb{Q}[x]$ about the multiplicative set 
$$
S_{\mathfrak{a}}=\mathbb{Q}[x]-(\mathfrak{p}_{1}\cup\cdots \cup\mathfrak{p}_{k}).
$$ 
The localization $S^{-1}_{\mathfrak{a}}\mathbb{Q}[x]$ of $\mathbb{Q}[x]$ is \textit{the set of all rational functions with denominators not sharing any non-trivial common factor with $a(x)$}. It can be easily shown that $S^{-1}_{\mathfrak{a}}\mathbb{Q}[x]=\cap_{i=1}^{k}S^{-1}_{\mathfrak{p}_{i}}\mathbb{Q}[x]$. The process of localization can be formalized as a mapping $j_{S_{\mathfrak{a}}}$ from $\mathbb{Q}[x]$ to $S^{-1}_{\mathfrak{a}}\mathbb{Q}[x]$.

\begin{defn}\label{eval}
Given $r(x), s(x), a(x)\in \mathbb{Q}[x]$ and $s(x),a(x)$ satisfying (\ref{a_alpha}) and (\ref{ax}), we define the evaluation of the fraction $\frac{r(x)}{s(x)}$ modulo $a(x)$ from $S^{-1}_{\mathfrak{a}}\mathbb{Q}[x]$ to $\mathbb{Q}[x]$ as
$$
\res{\frac{r(x)}{s(x)}}{a(x)} = \rem{(\alpha(x)r(x))}{a(x)},
$$ 
where $\alpha(x) s(x)\equiv 1 \textnormal{ mod }a(x)$. 
\end{defn}
The above function, in a way, is an evaluation process that results in a polynomial when another polynomial is (symbolically) set to zero. The key to the definition of the $eval$ function is the following ring isomorphism. 

\begin{lem}\label{isomorphism}
The rings $S_{\mathfrak{a}}^{-1}\mathbb{Q}[x]/S_{\mathfrak{a}}^{-1}\mathfrak{a}$ and $\mathbb{Q}[x]/\mathfrak{a}$ are isomorphic. 
\end{lem}

Suppose $h$ is the isomorphism, and $\eta$ and $\pi_{\mathfrak{a}}$ are as defined in the previous subsection we can now state the $eval$ function as a sequence of compositions
$$
 S_{\mathfrak{a}}^{-1}\mathbb{Q}[x] \xrightarrow[]{\pi_{S_{\mathfrak{a}}^{-1}\mathfrak{a}}} S_{\mathfrak{a}}^{-1}\mathbb{Q}[x]/S_{\mathfrak{a}}^{-1}\mathfrak{a}\xrightarrow[]{h^{-1}}\mathbb{Q}[x]/\mathfrak{a}\xrightarrow[]{\eta}\mathbb{Q}[x].
$$
By Bezout's identity, there exist $\alpha(x),\beta(x)$ such that $\alpha(x)s(x)=1-\beta(x)a(x)$ and observing that $\frac{-\alpha(x)r(x)\beta(x)a(x)}{1-\beta(x)a(x)}\in S_{\mathfrak{a}}^{-1}\mathfrak{a}$ we have 
\begin{eqnarray}
\nonumber \res{\frac{r(x)}{s(x)}}{a(x)} &=& \eta\circ h^{-1}\circ \pi_{S_{\mathfrak{a}}^{-1}\mathfrak{a}} \left(\frac{r(x)}{s(x)}\right) =\eta\circ h^{-1}\left(\frac{r(x)}{s(x)}+S_{\mathfrak{a}}^{-1}\mathfrak{a}\right)\\
\nonumber &=& \eta\circ h^{-1}\left(\frac{\alpha(x)r(x)}{1-\beta(x)a(x)}+\frac{-\alpha(x)r(x)\beta(x)a(x)}{1-\beta(x)a(x)}+S_{\mathfrak{a}}^{-1}\mathfrak{a}\right) 
\end{eqnarray}
Upon simplification and using the definitions of $h$ and $\eta$ we get 
\begin{eqnarray}
\nonumber \res{\frac{r(x)}{s(x)}}{a(x)} &=& \eta\circ h^{-1}\left(\alpha(x)r(x)+S_{\mathfrak{a}}^{-1}\mathfrak{a}\right) \\
\nonumber &=& \eta\left(\alpha(x)r(x)+\mathfrak{a}\right) = \rem{\alpha(x)r(x)}{a(x)}
\end{eqnarray}
In effect, the $eval$ function acting on $r(x)/s(x)$ with respect to $a(x)$ first takes the reciprocal of $s(x)$ mod $a(x)$, which is $\alpha(x)$ and takes its remainder by $a(x)$ after multiplying with $r(x)$. The following properties can be proved directly from the definitions. 
\begin{lem} \label{lem_res}
Given $a(x), r_{i}(x), s_{i}(x)\in \mathbb{Q}[x]$ and $\textnormal{gcd}(r_{i}(x),s(x))=1$ for $i=0,1$. The following are some properties of $eval$ function:
\begin{enumerate}
\item $\res{r_{0}(x)}{a(x)}=\rem{r_{0}(x)}{a(x)}$
\item $\res{\frac{r_{0}(x)r_{1}(x)}{s_{0}(x)s_{1}(x)}}{a(x)}=\rem{\left\{\res{\frac{r_{0}(x)}{s_{0}(x)}}{a(x)}\res{\\\frac{r_{1}(x)}{s_{1}(x)}}{a(x)}\right\}}{a(x)}$
\item $\res{\frac{r_{0}(x) - p_{0}(x)a(x)}{s_{0}(x)- q_{0}(x) a(x)}}{a(x)}=\res{\frac{r_{0}(x)}{s_{0}(x)}}{a(x)}$ \textnormal{for} $p_{0}(x),q_{0}(x)\in \mathbb{Q}[x].$
\end{enumerate}
\end{lem}

%\begin{exmp}
%To compute $\res{\frac{1}{x+1}}{x^{2}+1}$. We can follow the above lemma to obtain 
%\begin{eqnarray}
%\nonumber \res{\frac{1}{x+1}}{(x^{2}+1)} &=& \res{\frac{1}{x+1}\times \frac{x-1}{x-1}}{(x^{2}+1)} \\
%\nonumber &=& \res{\frac{x-1}{x^{2}-1}}{(x^{2}+1)}. 
%\end{eqnarray}  
%By Lemma \ref{lem_res}(3) we substitute $-1$ for $x^2$ to get
% $$
%\res{\frac{1}{x+1}}{x^{2}+1}=\res{\frac{x-1}{-1-1}}{(x^{2}+1)}
%$$
%and using Lemma \ref{lem_res}(1) we obtain
% $$
%\res{\frac{1}{x+1}}{x^{2}+1}=\frac{1}{2}(1-x).
%$$
%\end{exmp}	

\subsection{Fourier Analysis of the $eval$ Function}\label{subsec_FA}
The $eval$ function behaves in a satisfying  manner under the action of Fourier Transform. Our goal in this section is to study the finite Fourier series of the term $h(x)/(1-x^{b})$ appearing in the $q$-partial fraction, where $h(x)=\res{f(x)/g(x)}{\Psi_{b}(x)}$ and $f(x),g(x)\in \mathbb{Q}[x]$ and $g(x)$ relatively prime to $\Psi_{b}(x)$. For an excellent discussion on finite Fourier series see \cite{Beck}.

\begin{lem}\label{eval_eval}
Let $f(x),g(x)\in \mathbb{Q}[x]$ and $\textnormal{gcd}(g(x),\Psi_{b}(x))=1$. Suppose
$$
h(x)=\res{\frac{f(x)}{g(x)}}{\Psi_{b}(x)}.
$$
Then, we have $h(\xi)=f(\xi)/g(\xi)$ for $\xi= e^{2\pi i/b}$.
\end{lem}
\begin{proof}
As $\textnormal{gcd}(g(x),\Psi_{b}(x))=1$, by Bezout's identity, there exist $\alpha(x),\beta(x)\in \mathbb{Q}[x]$ such that 
\begin{equation}\label{g_psi}
\alpha(x)g(x)+\beta(x)\Psi_{b}(x)=1.
\end{equation}
Substituting $\xi=e^{2\pi i/b}$ both sides of (\ref{g_psi}), we get $\alpha(\xi)g(\xi)=1.$\\
Also, 
\begin{eqnarray}
\nonumber h(x)=\res{\frac{f(x)}{g(x)}}{\Psi_{b}(x)}&=& f(x)\alpha(x) \textnormal{ rem } \Psi_{b}(x)\\
\nonumber &=& f(x)\alpha(x)-q(x)\Psi_{b}(x),
\end{eqnarray}
for some $q(x)\in \mathbb{Q}[x]$. Hence, we have $h(\xi)=f(\xi)\alpha(\xi)=f(\xi)/g(\xi).$
\end{proof}

\begin{exmp}
Evaluating $g(x)=(-\frac{1}{b}x\Psi'_{b}(x))^{k} \textnormal{ rem }\Psi_{b}(x)$ at $\xi = e^{2\pi i/b}$ gives us 
$$
g(\xi)=(-\frac{1}{b}\xi\Psi'_{b}(\xi))^{k}=\frac{1}{(1-\xi)^{k}},
$$
since we have $g(x)=eval\left(\frac{1}{(1-x)^{k}};\ \Psi_{b}(x)\right)$ by  Lemma \ref{pow_k}.
\end{exmp}

Let $a(n)$ be a periodic function on $\mathbb{Z}$ with a period $b$. We denote it by a series $F(x)=\sum^{\infty}_{n=0}a(n)x^{n}$. Then, one can easily show that $F(x)$ has a generating function   
$$
F(x)=\frac{h(x)}{1-x^{b}}\quad \textnormal{  for }\quad h(x)=\sum^{b-1}_{j=0}a(k)x^{j}.
$$
The finite Fourier series expansion is given by 
$$
a(n)=\frac{1}{b}\sum_{j=0}^{b-1}h(\xi^{j})\xi^{-nj},
$$
with $\xi=e^{2\pi i/b}$. We also denote the zeroth Fourier coefficient by
$$
\textnormal{avg}(h)=\frac{1}{b}\sum^{b-1}_{j=0}a(j).
$$
Let $h(x)=\res{f(x)/g(x)}{\Psi_{b}(x)}$. By Lemma \ref{eval_eval} and the finite Fourier series expansion we deduce 
\begin{equation}
F(x)=\frac{h(x)}{1-x^{b}} = \sum_{n=0}^{\infty}\left(\textnormal{avg}(h)+\frac{1}{b}\sum^{b-1}_{j=1}\frac{f(\xi^{j})}{g(\xi^{j})}\xi^{-jn}\right)x^{n}.
\end{equation}
Thus by the formal power series $F(x)=\sum^{\infty}_{n=0}a(n)x^{n}$ 
\begin{equation}\label{FFS}
a(n)=\textnormal{avg}(h)+\frac{1}{b}\sum^{b-1}_{j=1}\frac{f(\xi^{j})}{g(\xi^{j})}\xi^{-jn}.
\end{equation}

The following example is a connecting link between the $eval$ function and the Fourier-Dedekind sum given in \cite{Beck}.
\begin{exmp}\label{eval_FDS}
Let $n_{1},\cdots, n_{k}$ and $b$ be pairwise relatively prime positive integers and suppose $h(x)=\res{\frac{1}{(1-x^{n_{1}})\cdots (1-x^{n_{k}})}}{\Psi_{b}(x)}$. Then the $n^{th}$ term in the power series of $h(x)/(1-x^{b})$ is 
$$
a(n)=\textnormal{avg}(h)+\frac{1}{b}\sum^{b-1}_{j=1}\frac{\xi^{-jn}}{(1-\xi^{jn_{1}})\cdots (1-\xi^{jn_{k}})}.
$$ 
\end{exmp}

\subsection{Extended Cover-Up Method}\label{subsec_ECU}
Given $p_{1}(x),\cdots,p_{n}(x)\in \mathbb{Q}[x]$ non-trivial polynomials that are pairwise relatively prime i.e., $\textnormal{gcd}(p_{i}(x),p_{j}(x))=1$ if $i\neq j$. Algebraically speaking, a partial fraction decomposition of the rational function of the form 
$$
\frac{1}{p_{1}(x)\cdots p_{n}(x)}
$$
can be viewed as a Bezout's identity expressed in terms of the functions
\begin{equation}\label{gj}
g_{j}(x)=\prod_{\substack{i=1\\ i\neq j}}^{n}p_{1}(x),\quad\quad 1\leq j\leq n,
\end{equation}
which are coprime, $\textnormal{gcd}(g_{1}(x),\cdots, g_{n}(x))=1$. So, by Bezout's identity there exist $a_{1}(x),\cdots,a_{n}(x)$ polynomials in $\mathbb{Q}[x]$ such that 
$$
a_{1}(x)g_{1}(x)+\cdots+a_{n}(x)g_{n}(x)=1
$$
which translates by the definition of $g_{j}(x)$ given in (\ref{gj}) to partial fraction form as 
\begin{equation}\label{pf_basic}
\frac{1}{p_{1}(x)\cdots p_{n}(x)}=\frac{a_{1}(x)}{p_{1}(x)}+\cdots+\frac{a_{n}(x)}{p_{n}(x)}
\end{equation}

Now, suppose $f(x)\in \mathbb{Q}[x]$ with $\textnormal{deg }f(x)<\textnormal{deg }(p_{1}(x)\cdots p_{n}(x))$ and the $\textnormal{gcd}(f(x),p_{i}(x))=1$ for all $i=1,2,\cdots,n$. We have by Euclid's Division Theorem
$$
a_{i}(x)f(x)=p_{i}(x)q_{i}(x)+r_{i}(x), \textnormal{  deg }r_{i}(x)<\textnormal{ deg }p_{i}(x)
$$
where $q_{i}(x)$ and $r_{i}(x)$ are the quotient and remainder respectively of $a_{i}(x)$  when divided by $p_{i}(x)$. 

\begin{lem} \label{lem_rem}
Suppose the equation 
\begin{equation}\label{lem_bez1}
b_{1}(x)g_{1}(x)+\cdots+b_{n}(x)g_{n}(x)=f(x)
\end{equation}
holds for $\textnormal{deg }f(x)<\textnormal{deg }p_{1}(x)\cdots p_{n}(x)$ . Then, the equation 
\begin{equation}\label{bez_0}
(\rem{b_{1}(x)}{p_{1}(x)})g_{1}(x)+\cdots+(\rem{b_{n}(x)}{p_{n}(x)})g_{n}(x)=f(x)
\end{equation}
also holds. 
\end{lem} 
\begin{proof} To prove (\ref{bez_0}) it suffices to show that the quotients of $b_{i}(x)$ when divided by $p_{i}(x)$ add up to zero.  That is, suppose we have
$$
\rem{b_{i}(x)}{p_{i}(x)}=b_{i}(x)-p_{i}(x)q_{i}(x)
$$
where $q_{i}(x)$ is the quotient. Substituting $q_{i}(x)$ into (\ref{bez_0}) gives us
$$
(b_{1}(x)-p_{1}(x)q_{1}(x))g_{1}(x)+\cdots+(b_{n}(x)-p_{n}(x)q_{n}(x))g_{n}(x)=1
$$
by (\ref{gj}) we have
$$
(q_{1}(x)+\cdots+q_{n}(x))\prod_{i=1}^{n}p_{i}(x)=0.
$$ 
Since $p_{i}(x)$ is non-zero polynomial, for $i=1,\cdots,n$, the equation (\ref{bez_0}) holds. 
\end{proof}

The advantage of having the above lemma is in simplifying the partial fraction expansion for complicated numerators. Now we prove our fundamental lemma that connects Bezout's identity and the $eval$ function. 

\begin{lem}\label{lem_bez2}[Fundamental Lemma]
Suppose $g_{j}(x)$ for $j=1,\cdots,n$ satisfy (\ref{gj}) and 
\begin{equation}\label{linear_eqn}
a_{1}(x)g_{1}(x)+\cdots+a_{n}(x)g_{n}(x)=f(x)
\end{equation}
holds for $\textnormal{deg }f(x)<\textnormal{deg }\prod_{j=1}^{n}p_{j}(x)$ and $\textnormal{gcd}(f(x),p_{i}(x))=1$ for all $i=1,2,\cdots,n$. Then, the following equation also holds:
\begin{equation}\label{bez_1}
\res{\frac{a_{1}(x)}{f(x)}}{p_{1}(x)}g_{1}(x)+\cdots+\res{\frac{a_{n}(x)}{f(x)}}{p_{n}(x)}g_{n}(x)=1.
\end{equation}
\end{lem} 
\begin{proof}
As $\textnormal{gcd}(f(x),p_{i}(x))=1$ for $1\leq i \leq n$ there exist $\alpha_{i}(x),\beta_{i}(x)\in \mathbb{Q}[x]$ such that
$$
\alpha_{i}(x)p_{i}(x)+\beta_{i}(x)f(x)=1.
$$ 
By the definition of $eval$ function it suffices to prove the equation (\ref{bez_1}) with the factors $\res{\frac{a_{i}(x)}{f(x)}}{p_{i}(x)}$ replaced with $\rem{a_{i}(x)\beta_{i}(x)}{p_{i}(x)}$. Dividing both sides of the equation (\ref{linear_eqn}) by $f(x)$ and $\prod_{i=1}^{n}p_{i}(x)$ we have 
$$
\frac{1}{p_{1}(x)\cdots p_{n}(x)}=\frac{a_{1}(x)}{f(x)p_{1}(x)}+\cdots+\frac{a_{n}(x)}{f(x)p_{n}(x)}.
$$  
By multiplying in the numerator with the factor $(\alpha_{i}(x)p_{i}(x)+\beta_{i}(x)f(x))$, which essentially equals $1$, in the $i^{th}$ term we have
$$
\frac{1}{p_{1}(x)\cdots p_{n}(x)}=\sum^{n}_{i=1}\frac{a_{i}(x)(\alpha_{i}(x)p_{i}(x)+\beta_{i}(x)f(x))}{f(x)p_{i}(x)}.
$$ 
Upon simplifying and rearranging terms we have
$$
\frac{a_{1}(x)\alpha_{1}(x)+\cdots+a_{n}(x)\alpha_{n}(x)}{f(x)}=\frac{1}{p_{1}(x)\cdots p_{n}(x)}-\sum_{i=1}^{n}\frac{a_{i}(x)\beta_{i}(x)}{p_{i}(x)}
$$
i.e., left hand side is in $S_{\mathfrak{f}}^{-1}R$, the multiplicative set formed by $f(x)$ and yet has the denominator as $f(x)$. This is possible only if $a_{1}(x)\alpha_{1}(x)+\cdots+a_{n}(x)\alpha_{n}(x)=0$. Therefore, 
$$
\frac{1}{p_{1}(x)\cdots p_{n}(x)}=\frac{a_{1}(x)\beta_{1}(x)}{p_{1}(x)}+\cdots+\frac{a_{n}(x)\beta_{n}(x)}{p_{n}(x)}
$$
Using Lemma \ref{lem_rem} we have
$$
\frac{1}{p_{1}(x)\cdots p_{n}(x)}=\sum_{i=1}^{n}\frac{\rem{(a_{i}(x)\beta_{i}(x))}{p_{i}(x)}}{p_{i}(x)}=\sum_{i=1}^{n}\frac{\res{a_{i}(x)/f(x)}{p_{i}(x)}}{p_{i}(x)}
$$
Hence the lemma is proved. 
\end{proof}

Now, we are ready to prove the extended cover-up method alluded in the Introduction section.  

\begin{thm}[Extended Cover-Up Method]\label{cover_up}
Let $p_{1}(x),\cdots, p_{n}(x)$ be mutually coprime polynomials. Then, we have the partial fraction expansion 
$$
\frac{1}{p_{1}(x)\cdots p_{n}(x)}=\frac{k_{1}(x)}{p_{1}(x)}+\cdots+\frac{k_{n}(x)}{p_{n}(x)},
$$
where $k_{i}(x)=\res{\prod_{j=1,j\neq i}^{n}\frac{1}{p_{j}(x)}}{p_{i}(x)}$.
\end{thm}
\begin{proof}
The polynomials $g_{j}(x)=\prod^{n}_{i=1,i\neq j}p_{i}(x)$ for $j=1,\cdots,n$ are mutually coprime. So, there exist $a_{j}(x)\in \mathbb{Q}[x]$ for $1\leq j\leq n$ such that 
$$
a_{1}(x)g_{1}(x)+\cdots+a_{n}(x)g_{n}(x)=1,
$$ 
which translates to the partial fraction form as 
\begin{equation}\label{pf_basic1}
\frac{1}{p_{1}(x)\cdots p_{n}(x)}=\frac{a_{1}(x)}{p_{1}(x)}+\cdots+\frac{a_{n}(x)}{p_{n}(x)}.
\end{equation}
None of the polynomials $a_{j}(x)$ in (\ref{pf_basic1}) are zero. Otherwise, if $a_{i}(x)$ is zero for some $i$ then the left hand side of (\ref{pf_basic1}) does not belong to $S^{-1}_{\mathfrak{p}_{i}}\mathbb{Q}[x]$ but the right hand side belongs to $S^{-1}_{\mathfrak{p}_{i}}\mathbb{Q}[x]$;  a contradiction. 

By Lemma \ref{lem_bez2} and as each $a_{i}(x)$ is non-zero polynomial we can write the above equation as  
$$
\sum^{n}_{i=1}\res{\frac{a_{i}(x)}{a_{1}(x)g_{1}(x)+\cdots+a_{n}(x)g_{n}(x)}}{p_{i}(x)}g_{i}(x)=1
$$ 
By Lemma \ref{lem_res}(3) we have 
$$
\sum^{n}_{i=1}\res{\frac{a_{i}(x)\mod p_{i}(x)}{a_{1}(x)g_{1}(x)+\cdots+a_{n}(x)g_{n}(x)\mod p_{i}(x)}}{p_{i}(x)}g_{i}(x)=1
$$ 
As $g_{j}(x)\mod p_{i}(x)=0$ if $i\neq j$ and $a_{j}(x)$ non-zero for all $j=1,\cdots,n$
$$
\res{\frac{1}{g_{1}(x)}}{p_{1}(x)}g_{1}(x)+\cdots+\res{\frac{1}{g_{n}(x)}}{p_{n}(x)}g_{n}(x)=1
$$
Hence the result is proved. 
\end{proof}

\begin{exmp}
We determine the partial fraction of $\frac{1}{(x^2+a^2)(x^3+b^3)}$, for arbitrary but fixed $a,b>0$, using our method. Let the partial fraction be given by
$$
\frac{1}{(x^2+a^2)(x^3+b^3)} = \frac{a_{1}(x)}{(x^2+a^2)}+\frac{a_{2}(x)}{(x^3+b^3)}
$$
we now evaluate for $a_{1}(x)$ and $a_{2}(x)$. 
\begin{eqnarray}
\nonumber a_{1}(x)&=&\res{\frac{1}{x^{3}+b^{3}}}{x^{2}+a^{2}} = \res{\frac{1}{x^{3}+b^{3}}\times \frac{x^{3}-b^{3}}{x^{3}-b^{3}}}{x^{2}+a^{2}}\\
\nonumber &=& \res{\frac{x^{3}-b^{3}}{x^{6}-b^{6}}}{x^{2}+a^{2}}= \frac{a^{2}x+b^{3}}{a^{6}+b^{6}}.
\end{eqnarray}
\begin{eqnarray}
\nonumber a_{2}(x)&=&\res{\frac{1}{x^{2}+a^{2}}\times \frac{x^{4}-a^{2}x^{2}+a^{4}}{x^{4}-a^{2}x^{2}+a^{4}}}{x^{3}+b^{3}}= \frac{-a^{2}x^{2}-b^{3}x-a^{4}}{b^{6}+a^{6}}
\end{eqnarray}
Hence, the partial fraction is 
$$
\frac{1}{(x^2+a^2)(x^3+b^3)}=\frac{1}{a^{6}+b^{6}}\left(\frac{a^{2}x+b^{3}}{x^{2}+a^{2}}+\frac{-a^{2}x^{2}-b^{3}x-a^{4}}{x^{3}+b^{3}}\right)
$$
Note that one can factorize $(x^{3}+b^{3})$ and follow the above method to perform further partial fraction decomposition. 
\end{exmp}
%
%\[\begin{tikzcd}
%A_{f} \arrow{r}{\varphi_f} \arrow[swap]{d}{\varrho_x^{f}} & B_{g}\arrow{d}{\varrho_x^{g}}\\
%A_{x} \arrow{r}{\varphi_{y}} & B_{y}
%\end{tikzcd}
%\]

\section{Partial Fraction Identities of Cyclotomic Polynomials}\label{CP}

In this section, we prove two main partial fraction results. Our first aim is to prove a partial fraction identity for a proper fraction of the general form
$$
f(x)\prod_{j=1}^{k}\frac{1}{\Phi_{d_{j}}(x)}=\frac{f(x)}{\Phi_{d_{1}}(x)\cdots\Phi_{d_{k}}(x)}, 
$$ 
for $f(x)\in \mathbb{Q}[x]$ and distinct positive integers $d_{1},\cdots, d_{k}$. Subsequently, state some corollaries and examples of this result. The strategy we employ to prove the first identity is to apply a `derivative' on the factorization of a polynomial to express it in the Bezout's identity form and then apply the fundamental lemma proved in the previous section. 

The derivative operator of our interest is the Cauchy-Euler operator of the form $(xD-m)$. These operators play a very crucial role in generating function theory and combinatorics due to their good algebraic properties. 

We now consider the partial fractions of  
$$
f(x)\prod_{j=1}^{k}\frac{1}{\Phi_{d_{j}}(x)}, \textnormal{  for distinct and positive } d_{1},\cdots, d_{k}.
$$ 
Let $m=\textnormal{lcm}(d_{1},\cdots,d_{k})$ and $g(x)$ be a multiplier such that 
\begin{equation}\label{m_equation}
g(x)\Phi_{d_{1}}(x)\cdots \Phi_{d_{k}}(x)=1-x^{m}.
\end{equation}
i.e., $g(x)=\prod_{t\in A}\Phi_{t}(x)$ where $A$ is the set of all divisors of $m$ excluding $\{d_{1},\cdots,d_{k}\}$. Also, suppose $f(x)\in \mathbb{Q}[x]$ with degree less than $\prod_{j=1}^{k}\Phi_{d_{j}}(x)$ and $\textnormal{gcd}(f(x),\Phi_{d_{j}}(x))=1$ for all $j=1,\cdots,k$.
Applying the operator $(xD-m)$ both sides of (\ref{m_equation}) we get 
$$
xg'(x)\prod_{j=1}^{k}\Phi_{d_{j}}(x)+\sum_{j=1}^{k}xg(x)\Phi'_{d_{j}}(x)\prod_{i\neq j}\Phi_{d_{i}}(x)-mg(x)\prod_{j=1}^{k}\Phi_{d_{i}}(x)= -m.
$$
Rearranging terms,
$$
\sum_{j=1}^{k}xg(x)\Phi'_{d_{j}}(x)\prod_{i\neq j}\Phi_{d_{i}}(x)= -m-(xg'(x)-m)\prod_{j=1}^{k}\Phi_{d_{j}}(x).
$$
Multiplying both sides by $f(x)$ and applying the Fundamental Lemma \ref{lem_bez2} we get 
$$
\sum_{j=1}^{k} \res{\frac{f(x)g(x)x\Phi'_{d_{j}}}{-m-(xg'(x)-m)\prod_{j=1}^{k}\Phi_{d_{j}}(x)}}{\Phi_{d_j}(x)}\prod_{i\neq j}\Phi_{d_{i}}(x)=f(x).
$$
By Lemma \ref{lem_res}(3) we have the final partial fraction decomposition 
$$
\sum_{j=1}^{k}\frac{\rem{(-f(x)g(x)x\Phi'_{d_{j}}(x))}{\Phi_{d_j}}(x)}{\Phi_{d_{j}}(x)}=mf(x)\prod_{j=1}^{k}\frac{1}{\Phi_{d_{j}}(x)}.
$$
The above discussion gives us the following main decomposition result. 
\begin{thm}\label{main1}
Let $d_{1},\cdots,d_{k}$ be distinct positive integers and $m$ be their least common multiple. Suppose $f(x)\in \mathbb{Q}[x]$ is a polynomial with $\textnormal{ deg }f(x)<\sum_{j=1}^{k}\phi(d_{j})$ and $f(x)$ does not have any of its roots as a $m^{th}$ root of unity. Then, we have
\begin{equation}
f(x)\prod_{j=1}^{k}\frac{1}{\Phi_{d_{j}}(x)}=\sum^{k}_{j=1}\frac{\rem{\left(-f(x)g(x)x\Phi'_{d_{j}}(x)\right)}{\Phi_{d_{j}}(x)}}{m\Phi_{d_{j}}(x)}
\end{equation}
where $g(x)=(1-x^{m})/(\prod_{j=1}^{k}\Phi_{d_{j}}(x))$.
\end{thm}

Theorem \ref{main1} is a very general case. Specifically, if we consider $\{d_{i}:d_{i}|m\}$ the set of all divisors of $m$ and $f(x)=x^{s}$ we have the following corollary. 
\begin{cor} For a positive integer $m$ and a non-negative integer $s$ with $0\leq s<m$
\begin{equation}\label{divisors}
\frac{x^s}{\Phi_{1}(x)}\prod_{d|m, d\neq 1}\frac{1}{\Phi_{d}(x)}= \frac{1}{m\Phi_{1}(x)}+\sum_{d|m,d\neq 1}\frac{(-x^{s+1}\Phi_{d}'(x)) \textnormal{ rem } \Phi_{d}(x)}{m\Phi_{d}(x)}.
\end{equation}
\end{cor}

\begin{rema}
For $s=0$ case by the result above we can easily deduce a direct \textit{derivative of logarithm} approach to arrive at the partial fraction identity
\begin{small}
$$
\frac{m}{1-x^{m}}=\left.y\frac{d}{dy}\log \prod_{d|m} \Phi_{d}(y)\right\vert_{y=1/x}. 
$$
\end{small}
\end{rema}

\begin{exmp}
We apply the Theorem \ref{main1} to perform partial fractions of 
$$
\frac{1}{\Phi_{1}(x)\Phi_{p}(x)\Phi_{q}(x)} \textnormal{  for distinct primes } p \textnormal{  and }q.
$$
Here $
g(x)=\frac{1-x^{pq}}{\Phi_{1}(x)\Phi_{p}(x)\Phi_{q}(x)}=\Phi_{pq}(x),
$ by Lenster's formula for $\Phi_{pq}(x)$\cite{Moree}:
$$
\Phi_{pq}(x)=\sum_{i=0}^{\rho}\sum_{j=0}^{\sigma}x^{ip+jq}-x\sum_{i=\rho+1}^{q-2-\rho}\sum_{j=\sigma+1}^{p-2-\sigma}x^{ip+jq}
$$
where $\rho$ and $\sigma$ are the unique non-negative integers for which $pq+1=(\rho+1)p+(\sigma+1)q.$ 
\begin{eqnarray}
\nonumber \frac{1}{\Phi_{1}(x)\Phi_{p}(x)\Phi_{q}(x)}&=&\frac{1}{pq\Phi_{1}(x)}+\frac{\rem{(-x\Phi_{pq}(x)\Phi'_{p}(x))}{\Phi_{p}(x)}}{pq\Phi_{p}(x)}\\
\nonumber & & \quad\quad\quad\quad\quad\quad+ \frac{\rem{(-x\Phi_{pq}(x)\Phi'_{q}(x))}{\Phi_{q}(x)}}{pq\Phi_{q}(x)}
\end{eqnarray}
\end{exmp}

\begin{exmp}
Let $s<2^{m}$. Then we have 
\begin{equation}\label{main_eq}
\frac{x^{s}}{\Phi_{1}(x)}\prod_{j=1}^{m}\frac{2}{\Phi_{2^{j}}(x)}=\frac{1}{\Phi_{1}(x)}+\sum_{j=1}^{m}\frac{(-1)^{\lfloor s/2^{j-1}\rfloor}2^{j-1}x^{s\%2^{j-1}}}{\Phi_{2^{j}}(x)}
\end{equation}
where the value $\lfloor \frac{m}{2^{j-1}}\rfloor$ is the largest positive integer less than or equal to $m/2^{j-1}$ and $(m\% 2^{j-1})$ stands for the remainder on dividing $m$ by $2^{j-1}$.

Specifically, for $m=11$ and $n=4$ the identity (\ref{main_eq}) becomes 
\begin{eqnarray}
\nonumber \frac{2^{4}x^{11}}{1-x^{16}}&=& \frac{2^{4}x^{11}}{(1-x)(1+x)(1+x^2)(1+x^4)(1+x^8)}\\ 
\nonumber                            &=& \frac{1}{1-x}-\frac{1}{1+x}-\frac{2x}{1+x^2}+\frac{2^2 x^{3}}{1+x^4}-\frac{2^3 x^3}{1+x^8}.
\end{eqnarray} 
It is interesting to note that the partial fraction expansion for the case $2^{m}$ follows the  binary representation of the integer $m$. In the binary representation of $m=(b_{n}\cdots b_{1})_{2}$ if $b_{j}$ is $0$ then the term corresponding to $(1+x^{2^{j-1}})$ has a positive sign and a negative sign otherwise. Thus, corresponding to the binary representation (with $n=4$) of $m=(11)_{10}=(1011)_{2}$ the signs in the above example are $(- - + -)$ in the reversed order. To explain the sign change, let us represent $m$ in binary representation. Suppose, $m=(b_{n}\cdots b_{1})_{2}$ then in modulo $x^{2^{k-1}}+1$ (i.e., substituting for $x^{2^{k-1}}$ by $-1$) we have 
\begin{eqnarray}
\nonumber x^{m} &=&x^{b_{n}2^{n-1}}x^{b_{n-1}2^{n-2}}\cdots x^{b_{k}2^{k-1}} x^{b_{k-1}2^{k-2}}\cdots x^{b_{1}2^{0}} \\
\nonumber &=& 1\ \ \times\ \  1\times \cdots  \cdots \times (x^{2^{k-1}})^{b_{k}} x^{b_{k-1}2^{k-2}}\cdots x^{b_{1}2^{0}} \\
\nonumber &=& \left\{
     \begin{array}{@{}l@{\thinspace}l}
       -x^{m \%2^{k-1}}  & \text{if   } b_{k}=1 \\
       \;\;\;x^{m \%2^{k-1}}  & \text{if   } b_{k}=0. \\
     \end{array}
   \right.
\end{eqnarray}
\end{exmp}

%\begin{exmp}
%For $m=6$ and $s=0$ in (\ref{divisors}) we have 
%\begin{eqnarray}
%\nonumber \frac{6}{1-x^{6}}&=& \frac{6}{(1-x)(1+x)(1+x+x^{2})(1-x+x^{2})}\\
%\nonumber &=& \frac{1}{1-x}+\frac{1}{1+x}+\frac{2+x}{1+x+x^{2}}+\frac{2-x}{1-x+x^{2}}
%\end{eqnarray}
%\end{exmp}

\subsection{Preparation Lemmas for the $q$-Partial Fractions}\label{di_lemma}
We prove two lemmas for deriving the partial fractions of (\ref{n_k}) 
$$
\frac{1}{\Phi_{1}(x)^{k}\Psi_{n}(x)}=\frac{1}{(1-x)^{k}\Psi_{n}(x)}\quad\quad \textnormal{ for } k\geq 1.
$$ 
First, we consider the case $k=1$. We start with applying the operator $xD$ both sides of the equation 
$$
(1-x)\Psi_{m}(x)=1-x^{m}
$$
to get 
$$
x\Psi'_{m}(x)(1-x)+(-x)\Psi_{m}(x)-m(1-x^{m})=-m.
$$
Applying the Fundamental Lemma \ref{lem_res} we get 
$$
\res{\frac{x\Psi'_{m}(x)}{-mx^{m}}}{\Psi_{m}(x)}(1-x)+\res{\frac{-x}{-mx^{m}}}{(1-x)}\Psi_{m}(x)=1,
$$
Simplifying we get the partial fraction  
\begin{equation}\label{pf_basic}
\frac{1}{(1-x)\Psi_{m}(x)}=\frac{1}{m(1-x)}+\frac{\rem{(-\frac{1}{m}x\Psi'_{m}(x))}{\Psi_{m}(x)}}{\Psi_{m}(x)}.
\end{equation}
Proceeding inductively we can prove the following result.
\begin{lem}\label{pow_k}
For $k\geq 1$ the following partial fraction holds:
\begin{equation}\label{pf_k}
\frac{1}{(1-x)^{k}\Psi_{m}(x)} =\frac{1}{m}\sum_{j=0}^{k-1}\frac{f_{j}^{(m)}(1)}{(1-x)^{k-j}}+\frac{f_{k}^{(m)}(x)}{\Psi_{m}(x)},
\end{equation}
where $f_{j}^{(m)}(x)=\rem{(-\frac{1}{m}x\Psi'_{m}(x))^{j}}{\Psi_{m}(x)}$ and $f_{0}^{(m)}(x)=1$.
\end{lem}
\begin{proof}
We prove by induction on $k$. The partial fraction (\ref{pf_basic}) is the base case $k=1$. Suppose the result holds for $k$.  Multiplying both sides of (\ref{pf_k}) by $\frac{1}{1-x}$we get 
$$
\frac{1}{(1-x)^{k+1}\Psi_{m}(x)} =\frac{\sum_{j=0}^{k-1}f_{j}^{(m)}(1)(1-x)^{j}}{m(1-x)^{k+1}}+\frac{f_{k}^{(m)}(x)}{(1-x)\Psi_{m}(x)}
$$
By performing partial fractions on the second term we have 
\begin{eqnarray}
\nonumber \frac{1}{(1-x)^{k+1}\Psi_{m}(x)} &=& \frac{\sum_{j=0}^{k-1}f_{j}^{(m)}(1)(1-x)^{j}}{m(1-x)^{k+1}}+\frac{f_{k}^{(m)}(1)}{m(1-x)}+\frac{f_{k+1}^{(m)}(x)}{\Psi_{m}(x)}\\
\nonumber &=& \frac{\sum_{j=0}^{k}f_{j}^{(m)}(1)(1-x)^{j}}{m(1-x)^{k+1}}+\frac{f_{k+1}^{(m)}(x)}{\Psi_{m}(x)}
\end{eqnarray}
Hence the result is proved. 
\end{proof}
We summarize the above result as:   
\begin{equation}\label{res_1x}
\res{\frac{1}{\Psi_{m}(x)}}{(1-x)^{k}}=\frac{1}{m}\sum_{j=0}^{k-1}f_{j}^{(m)}(1)(1-x)^{j}
\end{equation}
and
\begin{equation}\label{res_psi}
\res{\frac{1}{(1-x)^{k}}}{\Psi_{m}(x)} = f_{k}^{(m)}(x).
\end{equation}
where $f_{j}^{(m)}(x)=\rem{(-\frac{1}{m}x\Psi'_{m}(x))^{j}}{\Psi_{m}(x)}$ and $f_{0}^{(m)}(x)=1$.

We can also prove a preparation result for the general cyclotomic polynomial $\Phi_{a}(x)$ similar to Lemma \ref{pow_k}. 
\begin{lem}\label{cyclo_pow_k}
For $k\geq 1$ and $a>1$ the following partial fraction holds:
\begin{equation}\label{pf_cylco1}
\frac{1}{(1-x)^{k}\Phi_{a}(x)} =\frac{1}{\Phi_{a}(1)}\sum_{j=0}^{k-1}\frac{h^{(a)}_{j}(1)}{(1-x)^{k-j}}+\frac{h^{(a)}_{k}(x)}{\Phi_{a}(x)},
\end{equation}
where $h^{(a)}_{j}(x)=\rem{(-\frac{1}{a}x\tilde{\Phi}_{a}(x)\Phi'_{a}(x))^{j}}{\Phi_{a}(x)}, h^{(a)}_{0}(x)=1$, and $\\ \Phi_{1}(x)\Phi_{a}(x)\tilde{\Phi}_{a}(x)=1-x^{a}$.
\end{lem}
The proof of this lemma is again by induction on $k$ and is similar to that of the proof of Lemma \ref{pow_k}. As a consequence of this lemma we have     
\begin{equation}\label{res_2x}
\res{\frac{1}{\Phi_{a}(x)}}{(1-x)^{k}}=\frac{1}{\Phi_{a}(1)}\sum_{j=0}^{k-1}h^{(a)}_{j}(1)(1-x)^{j}.
\end{equation}

\begin{lem}\label{gcd}
Let $\textnormal{gcd}(m,n)=1$ and let $a,b>0$ satisfying $am-bn=1$. Then, 
\begin{equation}\label{n_m}
\res{\frac{1}{\Psi_{m}(x)}}{\Psi_{n}(x)} = \rem{\sum_{j=0}^{a-1}x^{jm\%n}}{\Psi_{n}(x)}
\end{equation}
and
\begin{equation}\label{m_n}
\res{\frac{1}{\Psi_{n}(x)}}{\Psi_{m}(x)} = \rem{-\sum_{j=0}^{b-1}x^{(jn+1)\%m}}{\Psi_{m}(x)}.
\end{equation}
\end{lem}
\begin{proof}
Observe that we can write 
$$
\res{\frac{1}{\Psi_{m}(x)}}{\Psi_{n}(x)} = \res{\frac{1-x}{1-x^{m}}}{\Psi_{n}(x)}.
$$
Upon multiplying and dividing by $(x^{(a-1)m}+\cdots+1)$ we have 
\begin{equation}\label{eq_am}
\res{\frac{1-x}{1-x^{m}}}{\Psi_{n}(x)}=\res{\frac{(1-x)(x^{(a-1)m}+\cdots+1)}{1-x^{am}}}{\Psi_{n}(x)}.
\end{equation}
By Lemma \ref{lem_res}(3), we can replace any one of the equivalent polynomials modulo $\Psi_{n}(x)$ at both numerator and the denominator. Moreover, by the Lemma \ref{cor_phi} when we perform modulo $\Psi_{n}(x)$ we can first work modulo $(1-x^{n})$ and replacing $x^{n}$ by $1$. As $am-bn=1$ we have 
$$
1-x^{am}=1-x^{am}x^{-bn}=1-x^{am-bn}=1-x  
$$
holds modulo $(1-x^{n})$.
Therefore, from (\ref{eq_am}) we can write 
\begin{eqnarray}
\nonumber \res{\frac{1}{\Psi_{m}(x)}}{\Psi_{n}(x)} &=& \res{\frac{(1-x)(x^{(a-1)m}+\cdots+1)}{1-x}}{\Psi_{n}(x)}\\
\nonumber &=& \res{x^{(a-1)m}+\cdots+1}{\Psi_{n}(x)}\\
\nonumber &=& \rem{\sum_{j=0}^{a-1}x^{jm\%n}}{\Psi_{n}(x)}.
\end{eqnarray}
We can derive (\ref{m_n}) in similar lines. 
\end{proof}

%For the case of two factors $\Psi_{m}(x)\Psi_{n}(x)$ the above result can be made more precise result. 
%
%\begin{thm}
%Let $m,n$ be two integers that are coprime and, let $m\geq 2$ and $n\geq3$. Consider the smallest possible positive numbers $a,b,\alpha,\beta$ for which we have $am-bn=1$ and $\alpha n- \beta m=1$. If $a>\beta$ then the partial fraction is given by 
%$$
%\frac{1}{\Psi_{m}(x)\Psi_{n}(x)}=\frac{\sum_{j=0}^{\alpha-1}x^{jn\%m}}{\Psi_{m}(x)}-\frac{\sum_{j=0}^{\beta-1}x^{(jm+1)\%n}}{\Psi_{n}(x)}
%$$
%else if $a<\beta$ we have the partial fraction 
%$$
%\frac{1}{\Psi_{m}(x)\Psi_{n}(x)}=\frac{\sum_{j=0}^{a-1}x^{jm\%n}}{\Psi_{n}(x)}-\frac{ \sum_{j=0}^{b-1}x^{(jn+1)\%m}}{\Psi_{m}(x)}.
%$$
%\end{thm}
%\begin{proof}
%It can be easily shown that $\beta=a$ holds only if $n\leq 2$ (the case we excluded). By the previous lemma one can easily check the following:
%$$
%\res{\frac{1}{\Psi_{m}(x)}}{\Psi_{n}(x)} = 
%\begin{cases} 
% \sum_{j=0}^{a-1}x^{jm\%n} & \textnormal{ if } a<\beta\\
% -\sum_{j=0}^{\beta-1}x^{(jm+1)\%n} & \textnormal{ if } a>\beta 
%\end{cases}
%$$
%and 
%$$
%\res{\frac{1}{\Psi_{n}(x)}}{\Psi_{m}(x)} = 
%\begin{cases} 
% -\sum_{j=0}^{b-1}x^{(jn+1)\%m} & \textnormal{ if } a<\beta \\
% \sum_{j=0}^{\alpha-1}x^{jn\%m} & \textnormal{ if } a>\beta 
%\end{cases}
%$$
%The above cases can be derived by observing from Corollary \ref{cor_psi} that to obtain the remainder by $\Psi_{n}(x)$ we need to replace $x^{n-1}$ by $(-1-\cdots-x^{n-2})$. Moreover, if $a<\beta$ then the term $x^{n-1}$ doesn't appear in $\sum_{j=0}^{a-1}x^{jm\%n}$.
%\end{proof}

\begin{exmp}
Let $m=263$ and $n=127$. We have $113\times 263-234\times 127=1$ and $29\times 127-14\times 263=1$. Thus, $a=113,b=234$ and $\alpha=29,\beta=14$. As $a>\beta$ the value $(n-1)=126$ doesn't occur in the exponent of the first term so the partial fraction by Lemma (\ref{gcd}) and the rule (\ref{rem_psi}) we have  
$$
\frac{1}{\Psi_{263}(x)\Psi_{127}(x)}=\frac{\sum_{j=0}^{28}x^{(127j)\%263}}{\Psi_{263}(x)}-\frac{\sum_{j=0}^{13}x^{(263j+1)\%127}}{\Psi_{127}(x)}.
$$
\end{exmp}

\subsection{$q$-Partial Fractions}\label{subsec_qpf}
Now, we consider a decomposition to the rational function 
$$
\frac{p(x)}{(1-x)^{m}(1-x^{n_{1}})\cdots(1-x^{n_{k}})},
$$ 
where $n_{1},\cdots, n_{k}$ are mutually coprimes, set $s=n_{1}+\cdots+n_{k}$ and $\textnormal{deg }(p(x))<m+s$ and $p(x)$ doesn't share a root with $(1-x^{n_{j}})$ for $1\leq j\leq k$. Then, we seek to find the $q$-partial fraction 
 $$
 \frac{p(x)}{(1-x)^{m+k}\Psi_{n_{1}}(x)\cdots \Psi_{n_{k}}(x)}=\frac{h_{0}(x)}{(1-x)^{m+k}} + \sum^{k}_{j=1}\frac{h_{j}(x)}{1-x^{n_{j}}}.
 $$
Suppose $p(x)=\tilde{p}(1-x)=\sum_{i=0}^{m+s-1}a_{i}(1-x)^{i}$ so that 
 $$
 h_{0}(x)=\left(\tilde{p}(1-x)\prod_{j=1}^{k}eval\left(\frac{1}{\Psi_{n_j}(x)}; (1-x)^{m+k-1}\right)\right) \textnormal{ rem }(1-x)^{m+k-1}
 $$
 can be expressed by using Lemma \ref{cor_xk} as  
 $$
 h_{0}(x)= \frac{1}{n_{1}\cdots n_{k}}\sum_{\substack{0\leq i+\alpha < m+k-1\\ \alpha=i_{1}+\cdots+i_{k}}} a_{i}f^{(1)}_{i_{1}}(1)\cdots f^{(k)}_{i_{k}}(1)(1-x)^{i+\alpha}.
 $$
 and further we have
\begin{eqnarray}
\nonumber h_{j}(x) &=&\res{\frac{1}{(1-x)^{m+k-1}\Psi_{n_{1}}(x)\cdots\widehat{\Psi_{n_{j}}(x)}\cdots\Psi_{n_{k}}(x)}}{\Psi_{n_{j}}(x)}\\
\nonumber &=& \left(\res{\frac{1}{(1-x)^{m+k-1}}}{\Psi_{n_{j}}(x)}\prod_{\substack{i=1\\ i\neq j}}^{k}\res{\frac{1}{\Psi_{n_{i}}(x)}}{\Psi_{n_{j}}(x)}\right)\textnormal{ rem }\Psi_{n_{j}}(x)\\
 \nonumber &=& \rem{\left(\left(-\frac{1}{n_{j}}x\Psi'_{n_{j}}(x)\right)^{m+k-1}\prod_{s=1,s\neq j}^{k}\sum_{i_{s}=0}^{a_{s}^{(j)}-1}x^{(i_{s}n_{s})\%n_{j}}\right)}{\Psi_{n_{j}}(x)}\label{separable_gj}.
\end{eqnarray}

\begin{rema}\label{avg_case}
In the computation above we considered only $m+k-1$ factors of $1/(1-x)$ while we actually have $m+k$ factors. The left out factor can be compensated by finally multiplying both sides of the partial fraction by $1/(1-x)$. A clear advantage in doing this is a reduction of one term for calculation. However, if necessary, we can also compute the formula with all $m+k$ factors where the change will be merely on the upper bound of the above formula from $m+k-1$ to $m+k$.
\end{rema}

\begin{thm}\label{main_qpf}
The following $q$-partial fraction holds
\begin{equation}
 \frac{p(x)}{(1-x)^{m}}\prod_{j=1}^{k}\frac{1}{1-x^{n_{j}}}=\sum^{m+k-2}_{j=0}\frac{c_{j}}{(1-x)^{m+k-j}} + \sum^{k}_{j=1}\frac{h_{j}(x)}{1-x^{n_{j}}},
\end{equation}
where $p(x)=\sum^{m+k-1}_{i=0}a_{i}(1-x)^{i}$,
$$
c_{j}=\frac{1}{n_{1}\cdots n_{k}}\sum_{i+i_{1}+\cdots+i_{k}=j} a_{i}f^{(1)}_{i_{1}}(1)\cdots f^{(k)}_{i_{k}}(1),
$$
and
$$
h_{j}(x) =  \rem{\left(p(x)f_{m+k-1}^{(j)}(x)\sum_{i_{1}=0}^{a_{1}^{(j)}-1}\cdots\sum_{i_{k}=0}^{a_{k}^{(j)}-1}x^{(i_{1}n_{1}+\cdots+i_{k}n_{k})\%n_{j}}\right)}{\Psi_{n_{j}}(x)}
$$
where $f_{k}^{(m)}(x)=(-\frac{1}{m}x\Psi'_{m}(x))^{k}, a_{j}^{(j)}=1$ and $a_{i}^{(j)}n_{i}-b_{i}^{(j)}n_{j}=1, a^{(j)}_{i}, b^{(j)}_{i}>0$ for $i\neq j$. 
\end{thm}

\section{Denumerants, Frobenius Number and Ehrhart Polynomials}\label{qPFD}
Let $n_{1},\cdots,n_{k}$ be a set of $k$ distinct positive integers satisfying 
$$
2\leq n_{1}<n_{2}<\cdots<n_{k}.
$$
Furthermore, we assume them to be mutually coprime, that is, $\textnormal{gcd}(n_{i},n_{j})=1$ for $i\neq j$. We denote the denumerants by $d(n_1,\cdots,n_k;t)$, the number of non-negative integer solutions $\alpha_{1},\cdots,\alpha_{k}$ of the equation
\begin{equation}\label{linear_equation}
\alpha_{1}n_{1}+\cdots+\alpha_{k}n_{k}=t.
\end{equation}
The Frobenius number is the largest number $t$ that cannot be expressed as (\ref{linear_equation}). Computing the Frobenius number is known to be NP-Hard under Turing reduction \cite{Alfonsin}. 

In order to compute the denumerants, we set $p(x)=1$ and $m=0$ in Theorem \ref{main_qpf} to have  
$$
  \prod_{j=1}^{k}\frac{1}{1-x^{n_{j}}}=\sum^{k-2}_{j=0}\frac{c_{j}}{(1-x)^{k-j}} + \sum^{k}_{j=1}\frac{h_{j}(x)}{n_{j}(1-x^{n_{j}})}.
$$
The coefficient  
\begin{equation}\label{coeff_standard}
c_{j}=\frac{1}{n_{1}\cdots n_{k}}\sum_{i_{1}+\cdots+i_{k}=j} f^{(1)}_{i_{1}}(1)\cdots f^{(k)}_{i_{k}}(1), \quad \textnormal{ for }0\leq j\leq k-2,
\end{equation}
and denote $h_{j}(x)=\sum^{n_{j}-1}_{i=0}c_{i}^{(j)}x^{i}$. By using the power series 
$$
\frac{1}{(1-x)^{l}}=\sum^{\infty}_{n=0}\begin{pmatrix}
n+l-1\\
n\\
\end{pmatrix}
x^{n},
$$
for $l=0,\cdots, (k-1)$, we obtain 
\begin{equation}\label{denum}
d(n_{1},\cdots,n_{k};t)=\sum^{k-2}_{j=0}c_{j}
\begin{pmatrix}
t+k-j-1\\
t\\
\end{pmatrix}
+\sum^{k}_{j=1}\frac{1}{n_{j}}c^{(j)}_{t\%n_{j}}.
\end{equation}
For example, a $q$-partial fraction of the generating function of partition by $3$ and $5$ is
$$
\frac{1}{(1-x^{3})(1-x^{5})}=\frac{1}{15(1-x)^{2}}+\frac{1-x}{3(1-x^{3})}+\frac{2 x^3-x^2+x+3}{5 \left(1-x^{5}\right)}.
$$
Thus by (\ref{denum}) 
$$
d(3,5;t) = 
\frac{1}{15}
\begin{pmatrix}
t+1\\
t\\
\end{pmatrix}+\frac{1}{3}c^{(1)}_{t\%3}+\frac{1}{5}c^{(2)}_{t\%5}.
$$
Assuming the multiplications of the polynomials is done using the FFT algorithm, the computational cost of setting up (\ref{denum}) is $O(n_{k}\log n_{k})$. Using the upper bound for the Frobenius number $\left(2n_{k-1}\lfloor\frac{n_{k}}{k}\rfloor-n_{k}\right)$ as given by Erd\H{o}s and Graham \cite{Erdos, Alfonsin}, the time taken for a linear search is $O(n_{k-1}n_{k})$ steps. Overall, to compute the Frobenius number one requires at the most $O(kn_{k}\log n_{k}+n_{k-1}n_{k})$ steps, which is polynomially bound by the numeric values of the input. Hence, we obtain a pseudo-polynomial time algorithm. 

\begin{exmp}
We want to determine the Frobenius number of the set of coprimes $\{9,17,31\}$. So, we consider the $q$-partial fraction of 
$$
F(x)=\frac{1}{(1-x^{9})(1-x^{17})(1-x^{31})}.
$$
Here $k=3$ and $n_{1}=9, n_{2}=17$ and $n_{3}=31$. We first compute $f_{i}^{(j)}(1)$ and obtain
$$
f^{(1)}_{0}(1)f^{(2)}_{0}(1)f^{(3)}_{0}(1)=1,\textnormal{ and } f^{(1)}_{1}(1)+f^{(2)}_{1}(1)+f^{(3)}_{1}(1) = 27.
$$ 
Furthermore, $a_{i}^{(j)}$ for each $j=1,2,3$ satisfy
$$
a_{j}^{(j)}=1 \textnormal{ and }a_{i}^{(j)}n_{i}-b_{i}^{(j)}n_{j}=1,\  a^{(j)}_{i}, b^{(j)}_{i}>0 \textnormal{ for }i\neq j
$$  
can be easily computed to obtain 
$$a_{1}^{(1)}= 1,\quad a_{2}^{(1)}=8,\quad a_{3}^{(1)}=7,\quad  a_{1}^{(2)}= 2,\quad  a_{2}^{(2)}=1, \quad  $$
$$a_{3}^{(2)}=11,\quad a_{1}^{(3)}=7,\quad a_{2}^{(3)}=11, \textnormal{ and } a_{3}^{(3)}=1.$$
Substituting for $f_{i}^{(j)}(1)$ and $a^{(j)}_{i}$ with simplifications yields:
\begin{small}
\begin{eqnarray}
\nonumber g_{0}(x)&=&\left(1+27(1-x)\right),\\
\nonumber g_{1}(x) &=&\left(-2-6x-3x^2-2x^3-3x^4-6x^5-2x^6\right),\\
\nonumber g_{2}(x)&=& \left(13+4x+7x^2+5x^3-2x^4+3x^5+3x^6-2x^7\right.\\
\nonumber & & \quad\quad\quad \left.+5x^8+7x^9+4x^{10}+13x^{11}-x^{13}+10x^{14}-x^{15}\right),\\
\nonumber g_{3}(x)&=&\left(14+13x-3x^{2}-3x^{3}+13x^{4}+14x^{5}+2x^{7}-11x^{8}\right.\\
\nonumber & &\quad\quad +23x^{9}+11x^{10}-16x^{11}+4x^{12}+9x^{13}-x^{14}+5x^{15}\\
\nonumber & & \quad\quad -4x^{16}+3x^{17}+26x^{18}+3x^{19}-4x^{20}+5x^{21}-x^{22}\\
\nonumber & & \quad\quad \left.+9x^{23}+4x^{24}-16x^{25}+11x^{26}+23x^{27}-11x^{28}+2x^{29}\right).
\end{eqnarray}
\end{small}
The partial fraction is 
$$
F(x)=\frac{1}{9\cdot17\cdot 31}\frac{g_{0}(x)}{(1-x)^{3}}+ \frac{1}{9}\frac{g_{1}(x)}{(1-x^{9})}+\frac{1}{17}\frac{g_{2}(x)}{(1-x^{17})}+\frac{1}{31}\frac{g_{3}(x)}{(1-x^{31})}.
$$
The denumerant is given by 
$$
d(9,17,31;t)=\frac{1}{4743}\left(\begin{pmatrix}
t+2\\
t\\
\end{pmatrix}+27\begin{pmatrix}
t+1\\
t\\
\end{pmatrix}\right)+\frac{1}{9}c^{(1)}_{t\%9}+\frac{1}{17}c^{(2)}_{t\%17}+\frac{1}{31}c^{(3)}_{t\%31}
$$
where $c^{(j)}_{t\%n_{j}}$ is the $(t\%n_{j})^{th}$ coefficient of $g_{j}(x)$.

By the upper bound of Erd\H{o}s and Graham we need to look for the Frobenius number only up to $309$. By a simple linear search we get the Frobenius number for the set $\{9,17,31\}$ as $73$.
\end{exmp}

The Ehrhart Polynomial corresponds to the number of non-negative integer solutions of the inequality 
$$
\alpha_{1}n_{1}+\cdots + \alpha_{k}n_{k}\leq t,
$$ 
corresponding to the polytope $\mathcal{P}$ formed by plane passing through the points $\{(\frac{1}{n_{1}},0,\cdots,0),\cdots,(0,\cdots,\frac{1}{n_{k}})\}$. We denote the number of integral points inside $\mathcal{P}$ by $Ehr_{\mathcal{P}}(t)$. The generating function for $Ehr_{\mathcal{P}}(t)$ is 
$$
 \frac{1}{(1-x)} \prod_{j=1}^{k}\frac{1}{1-x^{n_{j}}}=\sum^{k-2}_{j=0}\frac{c_{j}}{(1-x)^{k+1-j}} + \sum^{k}_{j=1}\frac{h_{j}(x)}{n_{j}(1-x^{n_{j}})}.
$$
This corresponds to the case $p(x)=1$ and $m=1$ in Theorem \ref{main_qpf}. Denoting $h_{j}(x)=\sum^{n_{j}-1}_{i=0}\hat{c}_{i}^{(j)}x^{i}$ we have

\begin{eqnarray}
\nonumber Ehr_{\mathcal{P}}(t)&=&\sum^{k-2}_{j=0}c_{j}
\begin{pmatrix}
t+k-j\\
t\\
\end{pmatrix}
+\sum^{k}_{j=1}\frac{1}{n_{j}}\hat{c}^{(j)}_{t\%n_{j}}\\
\nonumber &=& \sum_{j=2}^{k-j}\sum_{s=0}^{k-j-1}c_{j}\textnormal{stirl}(k-j-1,s)(1+t)^{s}+\sum^{k}_{j=1}\frac{1}{n_{j}}\hat{c}^{(j)}_{t\%n_{j}},
\end{eqnarray}
where $\textnormal{stirl(k,s)}$ is the Stirling number of first kind. By this formula one can notice that the Ehrhart polynomial is indeed a quasi-polynomial with the constant term being periodic as a function of $t$.

%and writing the term $\frac{g_{j}(x)}{(1-x)(1-x^{n_j})}$ as  
%  $$
% \frac{g_{j}(x)}{(1-x)(1-x^{n_j})} = \frac{g_{j}(x)-g_{j}(1)}{(1-x)(1-x^{n_j})}+\frac{g_{j}(1)}{(1-x)(1-x^{n_j})}
% $$
% Further, note that the term $g_{j}(x)-g_{j}(1)$ is divisible by $x-1$. Hence, to simplify the above expression we need to consider $q$-partial fractions of the second term. 
% 
% $\frac{g_{j}(x)-g_{j}(1)}{(1-x)(1-x^{n_j})}=\sum_{t=0}^{n_{j}-2}c_{t}^{(j)}(x^{t-1}+\cdots+1)=h_{j}(x).$
% and 
% $$
% \frac{g_{j}(1)}{(1-x)(1-x^{n_j})}=\frac{g_{j}(1)}{n_{j}(1-x)}+\frac{g_{j}(1)(-x\Phi'_{n_{j}}(x))rem\Phi_{n_{j}}(x)}{n_{j}(1-x^{n_{j}})}
% $$
% 
%So that we have
%$$
%\frac{1}{(1-x)(1-x^{n_{1}})\cdots(1-x^{n_{k}})}=\sum_{j=0}^{k-2}\frac{d_{j}}{(1-x)^{k-j+1}}+\sum_{j=1}^{k}\frac{h_{j}(x)}{1-x^{n_{j}}}+\sum_{j=1}^{k}\frac{g_{j}(1)}{n_{j}(1-x)}+\sum^{k}_{j=1}\frac{\tilde{h}_{j}(x)}{n_{j}(1-x^{n_{j}})}
%$$

\section{Reciprocity Theorems of the Generalized Fourier-Dedekind Sums}\label{sec_recip}
Consider the generating function  
\begin{equation}\label{Fx}
F(x)=\frac{p(x)}{(1-x)^{m}(1-x^{n_1})\cdots (1-x^{n_k})}=\sum^{\infty}_{n=0}a(n)x^{n}, 
\end{equation}
where $m\geq 0$, $n_{1},\cdots, n_{k}$ are mutually coprime and $\textnormal{deg }p(x)=d<s+m$ for $s=n_{1}+\cdots+n_{k}$. We consider a broader class of generating functions than Gessel \cite{Gessel} and Carlitz \cite{Carlitz} where they consider $p(x)=x^{r}$ and $p(x)=(1-x)$ respectively. Both the works only consider the case $k=2$.  

\begin{defn}\label{FDS}
For a positive integer $b$ that is coprime to $n_{j}$ for $1\leq j\leq k$, the generalized Fourier-Dedekind sum of $F(x)$ defined in (\ref{Fx}) for $k>0$ is 
\begin{equation}\label{gen_FDS}
S_{t}(n_{1},\cdots,n_{k}; b)=\frac{1}{b}\sum^{b-1}_{j=1}\frac{p(\xi)\xi^{jt}}{(1-\xi^{j})^{m}(1-\xi^{jn_{1}})\cdots(1-\xi^{jn_{k}})}
\end{equation}
and for $k=0$, we define, 
\begin{equation}\label{trivial_gen_FDS}
S_{t}(b)=\frac{1}{b}\sum^{b-1}_{j=1}\frac{p(\xi)\xi^{jt}}{(1-\xi^{j})^{m}}.
\end{equation}
where $\xi=e^{2\pi i/b}$. Here $k$ refers to the dimension of the sum.
\end{defn} 

Clearly, if we set $p(x)=1$ and $m=0$ in the Definition \ref{FDS} we recover, for $k>0$, the Fourier-Dedekind sum \cite{Beck}
\begin{equation}\label{classical_FDS}
s_{t}(n_{1},\cdots,n_{k}; b)=\frac{1}{b}\sum^{b-1}_{i=1}\frac{\xi^{it}}{(1-\xi^{in_{1}})\cdots(1-\xi^{in_{k}})}.
\end{equation}

By reciprocity law we mean identities for certain sums of generalized Fourier-Dedekind sums. In the literature, there are essentially two types of reciprocity laws: Zagier type (for $n=0$ case) and Rademacher type (for some $n\neq 0$ case) \cite{Beck}. The Zagier type reciprocity statement is essentially a comparison of the constant term in the formal power series expansion. On the other hand, the Rademacher reciprocity law relies on expressing (\ref{Fx}) as 
$$
F(1/x)=(-1)^{\alpha}x^{\beta}F(x),
$$
for suitable $\alpha,\beta$ integers. To ensure such relation holds one needs to assume $p(x)$ in (\ref{Fx}) to be palindromic, i.e., $p(1/x)=(-1)^{\delta}x^{-d}p(x)$. Our scheme for developing Rademacher reciprocity results is as follows: 
 \begin{enumerate}
 \item Perform partial fractions of $F(x)$ using the extended cover up method. 
 \item Compute the finite Fourier series for the periodic terms and defining the formula for the generalized Fourier-Dedekind sum. 
 \item Using $F(1/x)=(-1)^{k+m}x^{m+s}F(x)$ and comparing the $n^{th}$ coefficient in the formal power series $F(1/x)=\sum^{\infty}_{n=0}a(-n)x^{-n}$ we have 
 $$
 a(-n)=(-1)^{k+m}a(n-m-s).  
$$ 
 \item The reciprocity result is derived by $a(-n)=0$ $\textnormal{ for } 1\leq n<m+s$.
 \end{enumerate}
Note that the above steps only provide us an analytic result; we do not discuss the combinatorial meaning of the coefficient $a(-n)$. After deriving the reciprocity theorems we also demonstrate in our context the similarity of reciprocity law with the greatest common divisor.

Suppose the formal power series $F(x)=\sum^{\infty}_{n=0}a(n)x^{n}$. By Theorem \ref{main_qpf} and the equation (\ref{FFS}) we have 
\begin{equation}\label{an_term}
a(n) =  \textnormal{poly}(n) +\quad \sum^{k}_{j=1} S_{-n}(n_1,\cdots,\widehat{n_{j}},\cdots,n_{k};\ n_{j}),
\end{equation}
where $\widehat{}$ as before means dropping the corresponding term and 
\begin{equation}\label{poly1}
\textnormal{poly}(n)= \sum_{j=0}^{m+k-1}c_{j}\begin{pmatrix}n+m+k-j \\ n\end{pmatrix}.
\end{equation}
Also, note that the terms corresponding to the average of coefficients is included in the polynomial part of the expression. This could be done as given in Remark \ref{avg_case}.

\begin{thm}\label{zagier_type} 
For all pairwise relatively prime positive integers $n_{1},\cdots, n_{k}$,
$$
 \sum^{k}_{j=1} S_{0}(n_1,\cdots,\widehat{n_{j}},\cdots,n_{k};n_{j}) =p(0)-\textnormal{poly}(0)
$$
holds, where $c_{j}$ and $h_{j}(x)$ are defined as in Theorem \ref{main_qpf}, Remark \ref{avg_case} and $\textnormal{poly}(n)$ as in (\ref{poly1}).
\end{thm}
\begin{proof}
The constant term in the power series expansion of (\ref{Fx}) is $a(0)=p(0)$. Putting $n=0$ in the equation (\ref{an_term}) we have the result. 
\end{proof}

\begin{exmp}[Zagier]
Consider the set $A=\{n_1, \cdots, n_{k}\}$ of positive distinct integers that are mutually coprime and $k$ odd. Let
\begin{equation}\label{zagier_poly}
p(x)=(1+x^{n_{1}})\cdots(1+x^{n_{k}})+(1-x^{n_{1}})\cdots(1-x^{n_{k}}).
\end{equation}
 Up on a simplification we have 
 $$ 
 p(x)=2\sum_{S\subset A}x^{|S|} \textnormal{ where the subsets } S \textnormal{  of  } A \textnormal{ have an even cardinality.} 
 $$  
For $\xi=e^{2i\pi b}$, $p(\xi)=2(1+\xi^{n_{1}})\cdots\widehat{(1+\xi^{n_{j}})}\cdots(1+\xi^{n_{k}})$, and therefore the Zagier's higher dimensional Dedekind sum is
$$
Z(n_{1},\cdots,n_{k}; b)=\frac{2}{b}\sum^{b-1}_{i=1}\frac{(1+\xi^{in_{1}})\cdots\widehat{(1+\xi^{in_{j}})}\cdots(1+\xi^{in_{k}})}{(1-\xi^{in_{1}})\cdots\widehat{(1-\xi^{in_{j}})}\cdots(1-\xi^{in_{k}})}.
$$
Now, we can apply Theorem \ref{zagier_type} to obtain the reciprocity result with $p(x)$ given by (\ref{zagier_poly}) and $m=0$.
\end{exmp}

Now, we prove a Rademacher-like reciprocity theorem. Assume that $p(x)$ is palindromic, i.e., $p(1/x)=(-1)^{\delta}x^{-d}p(x),\  d = \textnormal{deg} p(x)$. To obtain a reciprocity result we consider 
$$
F(1/x)=(-1)^{\delta+k+1}x^{s+m-d}F(x).
$$
Then by comparing the $n^{th}$ term of the formal power series we have 
$$
a(-n)=\left\{\begin{array}{@{}l@{\thinspace}l}
       0 & \text{   if   } 1 \leq n \leq s+m-d-1 \\
      (-1)^{\delta+k+1}a(n-m-k-d)   & \text{   if   } s+m-d\leq n. \\
     \end{array}\right.
$$

\begin{thm}[Reciprocity Theorem]\label{k_dim_recip_thm}
For all pairwise relatively prime positive integers $n_{1},\cdots, n_{k}$, for $1\leq n \leq s+m-d-1$, the equation  
\begin{equation}
\nonumber \sum^{k}_{j=1} S_{n}(n_1,\cdots,\widehat{n_{j}},\cdots,n_{k};\ n_{j})=-\textnormal{poly}(-n)
\end{equation}
holds, where $c_{j}$ and $h_{j}(x)$ are defined as in Theorem \ref{main_qpf}, Remark \ref{avg_case} and $\textnormal{poly}(n)$ given in (\ref{poly1}).
\end{thm}
\begin{proof}
Replacing $n$ by $-n$ in (\ref{an_term}) and observing that   
$$
a(-n)=0,\quad \textnormal{     if     }\quad 1\leq n\leq s+m-d-1.
$$
Further, we using the identity 
\begin{equation}\label{comb_n}
(-1)^{l}\begin{pmatrix}
-n+\lambda\\
l\\
\end{pmatrix}=\begin{pmatrix}
n+l-1-\lambda\\
l\\
\end{pmatrix}, 
\end{equation}
for $n,\lambda\in \mathbb{Z} \textnormal{ and } l\in \mathbb{Z}^{+},$ we obtain the required result. 
\end{proof}
Similar to the Fourier-Dedekind sum, the generalized Fourier-Dedekind sum also enjoys further interesting properties.
\begin{enumerate}
\item $S_{t}(n_{1},\cdots,n_{k};b )$ is a rational number and is symmetric in $n_{1},\cdots, n_{k}$. 
\item $S_{t}(n_{1},\cdots,n_{k};b )$ only depends on $n_{i}\mod b$. 
\item $S_{t}(\lambda n_{1},\cdots,\lambda n_{k};b )=S_{t}(n_{1},\cdots,n_{k};b )$ if $\lambda$ is an integer prime to $b$.
\item $S_{t}(n_{1},\cdots,n_{k};b )=s_{t}(n_{1},\cdots,n_{k};b )*_{t}S_{t}(b)$. 
\end{enumerate}
We leave the proof of above properties to the reader. Property (4) follows from the Convolution Theorem for Finite Fourier Series  \cite[Theorem 7.10]{Beck}. 

In \cite{Beck}, the authors discuss the relation between the reciprocity law and the greatest common divisor. This relation becomes quite evident in our work when one writes the reciprocity theorems directly in terms of the coefficients as giving in (\ref{denum}). For the case $k=2$, i.e., $
\frac{1}{(1-x^{n_{1}})(1-x^{n_{2}})}$ with $n_{1}$ and $n_{2}$ relatively prime, the reciprocity theorem is 
$$
\frac{c^{(n_{1})}_{-n\%n_{1}}}{n_{1}}+\frac{c^{(n_{2})}_{-n\%n_{2}}}{n_{2}}+\frac{1}{n_{1}n_{2}}=0.
$$
Regrouping terms we have the Bezout's identity (a gcd relation)  
\begin{equation}\label{gcd_recip}
(-c^{(n_{1})}_{-n\%n_{1}})\ n_{2}+(-c^{(n_{2})}_{-n\%n_{2}})\ n_{1}=1.
\end{equation}
The Sylvester's reciprocity result for $k=2$ case \cite[Lemma 1.7]{Beck} 
$$
d(n_{1},n_{2};n)+d(n_{1},n_{2};n_{1}n_{2}-n)=1,
$$
for $1\leq n< n_{1}n_{2}$, written in terms of coefficients and regrouping terms we get 
\begin{equation}\label{gcd_sylvester}
\left(\frac{c^{(1)}_{n\%n_{1}}+c^{(1)}_{-n\%n_{1}}}{-2}\right)n_{2}+\left(\frac{c^{(2)}_{n\%n_{2}}+c^{(2)}_{-n\%n_{2}}}{-2}\right)n_{1}=1.
\end{equation}
The equations (\ref{gcd_recip}) and (\ref{gcd_sylvester}) depict the gcd relation between $n_{1}$ and $n_{2}$ through the coefficients. Similarly, for the case $k=3$, i.e., $
\frac{1}{(1-x^{n_{1}})(1-x^{n_{2}})(1-x^{n_{3}})}$ with $n_{1},n_{2}$ and $n_{3}$ being relatively prime. The coefficients are given by 
$$
c_{0}=\frac{1}{n_{1}n_{2}n_{3}},\quad c_{1}=\frac{1}{n_{1}n_{2}n_{3}}\left(f^{(n_{1})}_{1}(1)+f^{(n_{2})}_{1}(1)+f^{(n_{2})}_{1}(1)\right)=\frac{s-3}{2n_{1}n_{2}n_{3}},
$$
where $s=n_{1}+n_{2}+n_{3}$. The reciprocity theorem implies 
$$
c^{(n_{1})}_{-n\%n_{1}}n_{2}n_{3}+c^{(n_{2})}_{-n\%n_{2}}n_{1}n_{3}+c^{(n_{3})}_{-n\%n_{3}}n_{1}n_{2}=\frac{(n-1)(s-n-1)}{2},
$$
which again is a gcd relation between $\{n_{1}n_{2},n_{2}n_{3},n_{1}n_{3}\}$.

\subsection{Zero-Dimensional Generalized Fourier-Dedekind Sum}\label{0_dim}
In this section, we will prove two reciprocity results on the $0$-dimensional generalized Fourier-Dedekind sums. Technically speaking, these results here are `trivial' as we consider only one element $m$. Nevertheless, in our generalization we still have a $p(x)/(1-x)^{k}$ extra factor, which makes it a non-trivial case. The first result we prove follows directly from Theorem \ref{k_dim_recip_thm}. The second result is a `factorization' reciprocity that expresses the $0$-dimensional generalized Fourier-Dedekind sum of $m$ as a $(d(m)-1)$-dimensional sum corresponding to all the non-trivial factors of $m$.

\begin{thm}[$0$-Dimensional Reciprocity]\label{trivial_recip}
Given the generating function  
$$
F(x)=\frac{p(x)}{(1-x)^{k}(1-x^{m})},
$$
where $k\geq 0, m\geq 1, \textnormal{ deg }p(x)=d<m+k$. Then the equation
\begin{equation}\label{0_case} 
S_{0}(m)=p(0)-\frac{1}{m}\sum_{0\leq i+j\leq k}\frac{(-1)^{i}D^{i}p(1)}{i!}f_{j}^{(m)}(1)
\end{equation}
holds for all positive integers $m$. Further, if $p(1/x)=(-1)^{-\delta}x^{-d}p(x)$ holds, then
\begin{equation}\label{n_case} 
S_{n}(m) = -\frac{1}{m}\sum_{0\leq i+j \leq k}(-1)^{k-i-j}\frac{(-1)^{i}D^{i}p(1)}{i!}f^{(m)}_{j}(1)\begin{pmatrix}n-1 \\ k-i-j \end{pmatrix}
\end{equation}
hold for $1\leq n<m+k-d$, where
\begin{equation}\label{trivial_FDS}
S_{n}(m)=\frac{1}{m}\sum_{s=1}^{m-1}\frac{p(\xi^{s})\xi^{ns}}{(1-\xi^{s})^{k}}.
\end{equation}
\end{thm}

A direct consequence of the above result is that the general trigonometric sum
$$
P(m)=\sum_{s=1}^{m-1}\frac{p(\xi^{s})}{(1-\xi^{s})^{k}}
$$ 
is a polynomial in $m$ provided $f^{(m)}_{k}(1)$ is a polynomial in $m$ (which indeed is, as shown in the next subsection). The question of being polynomial for the case $p(x)=1$ was posed by Duran (Problem E 3339 Amer. Math. Monthly, \cite{Duran}); Gessel \cite{Gessel}  gave a short proof for the expression to be a polynomial in $m$; in this work, we show by algebraic methods that $f^{(m)}_{k}(1)$ is a polynomial in $m$. Hence, providing a novel method to solve the problem. 

\begin{proof}[Proof of Theorem \ref{trivial_recip}]
Suppose the formal power series $F(x)=\sum^{\infty}_{n=0}a(n)x^{n}$. Considering the relation  
$$
F(1/x)=(-1)^{k+\delta+1}x^{m+k-d}F(x)
$$
and comparing the coefficients we have 
$$
a(-n)=\left\{\begin{array}{@{}l@{\thinspace}l}
       p(0) & \text{   if   } n=0 \\
       0 & \text{   if   } 1 \leq n \leq m+k-d-1 \\
      (-1)^{k+\delta+1}a(n+d-m-k)   & \text{   if   } m+k-d\leq n. \\
     \end{array}\right.
$$
Using the Taylor series expansion, one can write 
$$
p(x)=\sum^{k}_{j=0}b_{j}(1-x)^{j}=\sum_{j=0}^{k}\frac{(-1)^{j}D^{j}p(1)}{j!}(1-x)^{j},
$$ 
where $D^{j}p(x)$ is the $j^{th}$ derivative of $p(x)$. So, by Lemma \ref{lem_rem} and Lemma \ref{pow_k} we have the decomposition 
\begin{equation}\label{pow_k_gen}
\frac{p(x)}{(1-x)^{k+1}\Psi_{m}(x)} =\frac{1}{m}\sum_{0\leq i+j\leq k}\frac{b_{i}f^{(m)}_{j}(1)}{(1-x)^{k+1-i-j}}+\frac{(1-x)g_{k+1}(x)}{(1-x^{m})},
\end{equation}
where $g_{k+1}(x)=(p(x)f_{k+1}(x))\textnormal{ rem }\Psi_{m}(x)$. Then, by the finite Fourier series we have 
\begin{equation}\label{an_trivial}
a(n) =\frac{1}{m}\sum_{0\leq i+j\leq k}b_{i}f_{j}^{(m)}(1)\begin{pmatrix} n+k-i-j \\ n \end{pmatrix} +S_{-n}(m).
\end{equation}
The case $n=0$ can be obtained by substituting $a(0)=p(0)$ into (\ref{an_trivial}) and the identity (\ref{comb_n}) yields (\ref{n_case}) for the case $1\leq n<m+k-d$. 
\end{proof}

Let $m$ be a positive integer and $\{a_{1},\cdots, a_{d(m)-1}\}$ be the set of all divisors of $m$ excluding $1$, where $d(m)$ is the number of divisors of $m$. To obtain a reciprocity relation for $m$ and its divisors we investigate the partial fractions on both sides of the equation using Lemma \ref{pow_k} and Theorem \ref{main1},
\begin{equation}\label{think_twice}
\frac{p(x)}{(1-x)^{k}\Psi_{m}(x)}=\frac{p(x)}{(1-x)^{k}}\prod_{j=1}^{d(m)-1}\frac{1}{\Phi_{a_{j}}(x)}.
\end{equation}
Equating the corresponding periodic parts of the partial fractions in (\ref{think_twice}) we get   
\begin{equation}\label{periodic_part}
\frac{(1-x)f_{k}(x)}{1-x^{m}}=\sum^{d(m)-1}_{j=1}\frac{\Theta_{a_{j}}(x)g_{j}(x)}{1-x^{a_{j}}},
\end{equation}
where 
 \begin{eqnarray}
 \nonumber f_{k}(x)&=&\res{\frac{p(x)}{(1-x)^{k}}}{\Psi_{m}(x)} \textnormal{ and}\\
 \nonumber g_{j}(x)&=&\res{\frac{p(x)}{(1-x)^{k}\Phi_{a_{1}}(x)\cdots \widehat{\Phi_{a_{j}}(x)}\cdots \Phi_{a_{d(m)-1}}(x)}}{\Phi_{a_{j}}(x)}.
\end{eqnarray} 

By matching the $n^{th}$ term in the finite Fourier series on both sides of the equation (\ref{periodic_part}) for $\xi=e^{2\pi i/m}$ we get 
\begin{small}
\begin{equation}\label{think_twice_sums}
\frac{1}{m}\sum_{s=1}^{m-1}\frac{p(\xi^{s})\xi^{-ns}}{(1-\xi^{s})^{k}}=\sum_{j=1}^{m}\frac{1}{a_{j}}\sum_{\xi_{a_{j}} \in \Delta_{a_{j}}}\frac{p(\xi_{a_{j}})\Theta_{a_{j}}(\xi_{a_{j}})\xi_{a_{j}}^{-n}}{(1-\xi_{a_{j}})^{k}\Phi_{a_{1}}(\xi_{a_{j}})\cdots \widehat{\Phi_{a_{j}}(\xi_{a_{j}})}\cdots \Phi_{a_{m}}(\xi_{a_{j}})}
\end{equation}
\end{small}where $\Theta_{a}(x)$ refers to the inverse cyclotomic polynomial. In view of the right hand side we denote another generalized Fourier-Dedekind sum based on the cyclotomic polynomials for $b>0$ distinct from $d_{1},\cdots, d_{r}$:
\begin{equation}\label{cyclo_fds}
S^{\Phi}_{n}(d_{1},\cdots,d_{r}; b):=\frac{1}{b}\sum_{\xi\in \Delta_{b}}\frac{p(\xi)\Theta_{b}(\xi)\xi^{n}}{(1-\xi)^{m}\Phi_{d_{1}}(\xi)\cdots\Phi_{d_{r}}(\xi)}.
\end{equation}

Thus, by (\ref{think_twice_sums}), we have the generalized Fourier-Dedekind sum for a positive number expressed as a sum of generalized Fourier-Dedekind sums corresponding to its divisors. 
\begin{thm} 
For a positive integer $m$ and the divisors (excluding $1$) of m  $\{a_{1},\cdots,a_{d(m)-1}\}$, the equation 
$$
S_{-n}(m)=\sum_{j=1}^{d(m)-1}S^{\Phi}_{-n}(a_{1},\cdots,\widehat{a_{j}},\cdots,a_{d(m)-1};\ a_{j}).
$$
holds for all $n=0,1,2,\cdots$.
\end{thm}

\subsection{Computation of $f_{k}^{(m)}(1)$}\label{computation}
From Section \ref{qPFD} we observed that for the purpose of computing denumerants the values $f^{(m)}_{k}(1), k=0,1,2,\cdots,$ play a crucial role. As an application of the reciprocity theorem Theorem \ref{trivial_recip} we derive a formula for $f^{(m)}_{k}(1)$, where     
$$
f^{(m)}_{k}(x)= \left(-\frac{1}{m}x\Psi'_{m}(x)\right)^{k} \textnormal{ rem } \Psi_{m}(x).
$$
Towards this direction, we denote $g_{k}(x)=\left(-x\Psi'_{m}(x)\right)^{k} \textnormal{ rem } \Psi_{m}(x)$ (that is, $g_{k}(x)$ equals $f^{(m)}_{k}(x)$ without the $1/m^{k}$ factor) and show that $g_{k}(1)$ is a polynomial in $m$ with a factor $m^{k}$. Consequently, the $1/m^{k}$ in $f^{(m)}_{k}(1)$ cancels with the implicit $m^{k}$ factors in its numerator. For the purpose of proving this claim we use the well known fact that $m(m+1)/2$ is a factor of the sum of powers $\sigma_{n}(m)=\sum_{i=1}^{m}i^{n}$.  

\begin{lem}\label{mk_factor}
The functions $g_{k}(x)$ and $\left(xD\right)^{n}g_{k}(x)$ when evaluated at $x=1$ are polynomials in $m$ with zero as a root of multiplicity at least $k$.
\end{lem}
\begin{proof}
We prove by induction on $k$. The result holds for $k=1$ case. Indeed, by a direct calculation we have 
$$
g_{1}(x)=x^{m-2}+\cdots+(m-2)x+(m-1)
$$
which implies $g_{1}(1)=m(m-1)/2$ and
$$
\left(xD\right)^{n}g_{1}(x)=(m-2)^{n}\cdot 1 x^{m-2}+(m-3)^{n}\cdot 2 x^{m-3}+\cdots+1^{n}\cdot(m-2)x 
$$
implies
\begin{eqnarray}
\nonumber \left.\left(xD\right)^{n}g_{1}(x)\right\vert_{x=1}&=&(m-2)^{n}\cdot 1+(m-3)^{n}\cdot 2+\cdots+1^{n}\cdot (m-2)\\
\nonumber &=& \sum_{i=1}^{m-1}i^{n}(m-i)-\sigma_{n}(m-1)\\
\nonumber &=& (m-1)\sigma_{n}(m-1)-\sigma_{n+1}(m-1),
\end{eqnarray}
which is divisible by $m$.

For the induction step, suppose $g_{k}(x)=\alpha_{m-2}x^{m-2}+\cdots+\alpha_{1}x+\alpha_{0}$, where the coefficients $\alpha_{0},\cdots,\alpha_{m-2}$ are dependent on $m$ and $k$, and assume $
g_{k}(1)$ and $\left.\left(xD\right)^{n}g_{k}(x)\right\vert_{x=1}=\sum_{t=0}^{m-2}t^{n}\alpha_{t}$ have $m^{k}$ as a factor for all positive integers $n$.

Consider the $(k+1)$ case:
\begin{eqnarray}
\nonumber g_{k+1}(x)&=& \left(g_{k}(x)(-x\Psi'_{m}(x))\right) \textnormal{ rem }\Psi_{m}(x)\\
\nonumber &=& \left(g_{k}(x)(x^{m-2}+\cdots+(m-2)x+(m-1))\right)\textnormal{ rem }\Psi_{m}(x) \\
\nonumber &=& \left(\sum_{i,j=0,0}^{m-2,m-2}(m-1-j)\alpha_{i}x^{i+j}\right) \textnormal{ rem } \Psi_{m}(x)\\
\nonumber &=& \left(\sum_{\stackrel{i,j=0,0}{(i+j)\%m\neq(m-1)}}^{m-2,m-2}\alpha_{i}(m-1-j)x^{(i+j)\% m}\right)\\
\nonumber & & \quad +   \left(\sum_{\stackrel{i,j=0,0}{(i+j)\%m=(m-1)}}^{m-2,m-2}\alpha_{i}(m-1-j)(x^{m-1}-(1+x\cdots+x^{m-1}))\right),
\end{eqnarray}
here we substitute for $x^{m-1}$ by $(x^{m-1}-(1+\cdots+x^{m-1}))$ using the rule (\ref{rem_psi}). Simplifying we obtain
$$
g_{k+1}(x)=\sum_{i,j=0,0}^{m-2,m-2}\alpha_{i}(m-1-j)x^{(i+j)\% m}-(1+x+\cdots+x^{m-1})\sum_{t=1}^{m-2}t\alpha_{t}.
$$
and substituting $x=1$ 
\begin{eqnarray}
\nonumber g_{k+1}(x) \nonumber &=&\sum_{i,j=0,0}^{m-2,m-2}\alpha_{i}(m-1-j)-m\sum_{t=1}^{m-2}t\alpha_{t}\\
\nonumber &=&m\left(\sum_{i=0}^{m-2}\frac{m-1}{2}\alpha_{i}-\sum_{t=1}^{m-2}t\alpha_{t}\right)= m^{k+1}\times (\textnormal{polynomial in } m),
\end{eqnarray}
where the last statement holds by the induction hypothesis. 

Now we show that $\left.\left(xD\right)^{n}g_{k+1}(x)\right\vert_{x=1}$ also has $m^{k+1}$ as a factor. Observe that  
$
\left.\left(xD\right)^{n}g_{k}(x)\right\vert_{x=1}= \sum_{t=0}^{m-2}t^{n}\alpha_{t}
$
and 
\begin{eqnarray}
\nonumber \left.\left(xD\right)^{n}g_{k+1}(x)\right\vert_{x=1}&=&\sum_{t=0}^{m-2}\left(\sum_{i=1}^{t-2}i^{n}(t-1-i)+\sum_{i=k}^{m-1}i^{n}(m-1+t-i)\right)\alpha_{t}\\
\label{gk_1} & & \quad\quad\quad\quad\quad\quad + \sigma_{n}(m-1)\sum_{t=0}^{m-2}t^{n}\alpha_{t}
\end{eqnarray}

Fortunately, we don't need to determine a formula for the summations; instead, we just need to show that $m$ is a factor.  The inner sum in the first term of (\ref{gk_1}) can be simplified to
\begin{eqnarray}
\nonumber (m-1)t^{n}+\cdots+t(m-1)^{n}&=&\sum_{i=t}^{m-1}i^{n}(m-1+t-i)\\
\nonumber &=& (m-1+t)\left\{\sigma_{n}(m-1)-\sigma_{n+1}(t-1)\right\}\\
\nonumber & & \quad\quad\quad\quad -\left\{\sigma_{n+1}(m-1)-\sigma_{n+1}(t-1)\right\}.
\end{eqnarray}
Furthermore, setting $m=0$ we get 
$-(t-1)\sigma_{n+1}(t-1)+\sigma_{n+1}(t-1).$ Also, the second term is 
\begin{eqnarray}
\nonumber (t-2)\cdot 1^{n}+\cdots+1\cdot (t-2)^{n}&=&\sum_{i=1}^{t-1}i^{n}(t-i)-\sigma_{n}(t-1)\\
\nonumber &=&(t-1)\sigma_{n}(t-1)-\sigma_{n+1}(t-1)
\end{eqnarray}
Thus, the factor $(m-1)t^{n}+\cdots+t(m-1)^{n}+(t-2)\cdot 1^{n}+\cdots+1\cdot (t-2)^{n}$ vanishes when $m=0$. Therefore, the equation (\ref{gk_1}) is of the form $\left.\left(xD\right)^{n}g_{k}(x)\right\vert_{x=1}=m\left(\sum_{t=0}^{m-2}q(m,t)\alpha_{t}\right),$ where $q(m,t)$ is a polynomial in $m$ and $t$. Hence, by the induction hypothesis $\left.\left(xD\right)^{n}g_{k+1}(x)\right\vert_{x=1}$ has $m^{k+1}$ as a factor. 
\end{proof}

\begin{thm}
The sum
$$
P(m)=\sum_{s=1}^{m-1}\frac{p(\xi^{s})}{(1-\xi^{s})^{k}}
$$ 
is a polynomial in $m$, for $\xi=e^{2\pi i/m}$ and for any polynomial $p(x)\in \mathbb{Q}[x]$ with $\textnormal{deg } p(x)<k+m-1$.
\end{thm}
\begin{proof}
From the equation (\ref{0_case}) it suffices to only show that $f^{(m)}_{k}(1)$ is a polynomial in $m$, which in turn follows by $f^{(m)}_{k}(1)=\frac{1}{m^{k}}g_{k}(1)$ and Lemma \ref{mk_factor}. 
\end{proof}

By taking a finite Fourier series of the equation (\ref{pf_k}) in Lemma \ref{pf_k}  
$$
\frac{1}{(1-x)^{k+1}\Psi_{m}(x)}=\frac{1}{m}\sum^{k}_{j=0}\frac{f^{(m)}_{j}(1)}{(1-x)^{k+1-j}}+\frac{(1-x)f^{(m)}_{k+1}(x)}{(1-x^{m})}.
$$
and considering only the constant term in its formal power series we have 
\begin{equation}\label{sigma_k}
1=\frac{1}{m}\sum_{j=0}^{k}f^{(m)}_{j}(1)+\frac{1}{m}\sum^{m-1}_{j=1}\frac{1}{(1-\xi^{j})^{k}}.
\end{equation}
Hence, we deduce 
$$
f^{(m)}_{k}(1)=\sum_{j=1}^{m-1}\left(\frac{1}{(1-\xi^{j})^{k-1}}-\frac{1}{(1-\xi^{j})^{k}}\right).
$$
Gessel \cite{Gessel} gave an explicit formula for $\vartheta_{k}(m)=(-1)^{k}\sum \frac{1}{(1-\xi^{j})^{k}}$ using the Bernoulli numbers and the Stirling numbers of first kind. Using this notation we can write  
$$
f^{(m)}_{k}(1)=(-1)^{k-1}\vartheta_{k-1}(m)-(-1)^{k}\vartheta_{k}(m).
$$
and by the formula given in \cite{Gessel}
\begin{equation}\label{gessel}
(-1)^{k}\vartheta_{k}(n)=-\frac{n-1}{2}+\frac{1}{(k-1)!}\sum^{k}_{j=2}(-1)^{j}\textnormal{stirl}(k,j)\frac{B_{j}}{j}(n^{j}-1).
\end{equation}

\begin{thm}\label{fjm_polynomial}
\begin{equation}\label{fjm}
f^{(m)}_{k}(1)= \left\{\begin{array}{@{}l@{\thinspace}l}
       \frac{m-1}{2} & \text{   if   } k=1 \\
       \frac{1}{(k-1)!}\sum^{k}_{j=2}\frac{B_{j}}{j}\textnormal{stirl}(k-1,j-1)(m^{j}-1)  & \text{   if   } k\geq 2. \\
     \end{array}\right.
\end{equation}
\end{thm}
\begin{proof} 
We can perform a direct computation for $k=1$ and for $k\geq 2$ we use the identity
$$
\textnormal{stirl}(k-i,j-1)=\textnormal{stirl}(k,j)-(k-1)\textnormal{stirl}(k-1,j). 
$$
and the equation (\ref{gessel}) we obtain the formula. 
\end{proof}

It is interesting to note that (\ref{fjm}) is associated with the `degenerate Bernoulli number' \cite{Gessel}. Moreover, by a direct computation one can observe that $f_{k}^{(2)}(1)=1/2^{k}$ and as $B_{2i+1}=0$ for $i\geq 1$ the structure of $f_{k}^{(m)}(1)$ is   
$$
f^{(m)}_{k}(1)=\frac{(m^{2}-1)g(m^{2})}{3\cdot2^{k}g(4)}
$$
for some polynomial $g(x)$, of degree $\leq \frac{k-2}{2}$, and an even function.

We digress a bit and perform a similar analysis of the main term appearing in Lemma \ref{cyclo_pow_k}, namely $h^{(a)}_{j}(1)$. To obtain a computational formula for $h_{j}^{(a)}(1)$ we perform a Fourier Analysis on
$$
\frac{h^{(a)}_{k}(x)}{\Phi_{a}(x)}=\frac{\Theta_{a}(x)h^{(a)}_{k}(x)}{1-x^{a}}
$$
to obtain the $n^{th}$ term in the formal power series expansion is 
$$
\frac{1}{a}\sum^{a-1}_{j=0}\Theta_{a}(\xi)h^{(a)}_{k}(\xi)\xi^{-jn}=\frac{1}{a}\sum_{\eta\in \Delta_{a}}\frac{\Theta_{a}(\eta)}{(1-\eta)^{k}}\eta^{-n}
$$
with $\xi=e^{2\pi i/a}$ and we used $\Theta_{a}(\eta)=0$ for $\eta\notin \Delta_{a}$ and $h_{k}^{(a)}(\xi)=1/(1-\xi)^{k}$.
\\
Comparing the $n=0$ terms in the power series of both sides of the equation  (\ref{pf_cylco1}) we get 
$$
\Theta_{a}(0)=\frac{1}{\Phi_{a}(1)}\sum^{k-1}_{j=0}h_{j}^{(a)}(1)+\frac{1}{a}\sum_{\eta\in \Delta_{a}}\frac{\Theta_{a}(\eta)}{(1-\eta)^{k}}
$$
So, we can compute the value of $h_{k}^{(a)}(1)$ as a trigonometric sum  
$$
h^{(a)}_{k}(1)=\frac{\Phi_{a}(1)}{a}\left(\sum_{\eta\in \Delta_{a}}\frac{\Theta_{a}(\eta)}{(1-\eta)^{k}}-\sum_{\eta\in \Delta_{a}}\frac{\Theta_{a}(\eta)}{(1-\eta)^{k+1}}\right).
$$
To the best of our knowledge, there is no known explicit formula for the Ramanujan like sum
$$
\tau_{k}(a)=\sum_{\eta\in \Delta_{a}}\frac{\Theta_{a}(\eta)}{(1-\eta)^{k}},
$$
where $\Theta_{a}(x)$ is the inverse cyclotomic polynomial and $\Delta_{a}$ is the set of all $a^{th}$ primitive roots of unity. 

Returning to $f^{(m)}_{k}(1)$, the computation of the $q$-partial fractions could be done easily with the polynomials $f^{(m)}_{k}(1)$. In fact, for a $k$-factor $q$-partial fraction we need $f_{j}^{(m)}(1)$ for $0\leq j\leq (k-2)$. We list a few polynomials (computed in SageMath) above. 
%
%\begin{figure}[h]
%\centering 
%\includegraphics[scale=0.5]{table}
%%\caption{Scheme for Reciprocity Theorem}
%\end{figure}

\begin{table}
\hspace{-1cm}
\begin{tabular}{|p{0.3cm}|p{12cm}|}

\hline

$k$ & \hspace{4cm}$f^{(m)}_{k}(1)$ \\ \hline

$0$ & $1$ \\ [0.1cm]

$1$ & $\frac{m-1}{2}$ \\ [0.1cm]

$2$ & $\frac{m^{2}-1}{12}$ \\[0.1cm]

$3$ & $\frac{m^{2}-1}{24}$ \\[0.1cm]

$4$ & $-\frac{(m^{2}-1)(m^{2}-19)}{720}$ \\[0.1cm]

$5$ & $-\frac{(m^{2}-1)(m^{2}-9)}{480}$ \\[0.1cm]

$6$ & $\frac{(m^{2}-1)(2m^{4}-145m^{2}+863)}{60480}$ \\[0.1cm]

$7$ & $\frac{(m^{2}-1)(m^{2}-25)(2m^{2}-11)}{24192}$ \\[0.1cm]

$8$ & $  -\frac{(3m^6 - 497m^4 + 9247m^2 - 33953)(m^2 - 1)}{3628800}$ \\[0.1cm]

$9$ & $ -\frac{(3m^4 - 50m^2 + 167)(m^2 - 1)(m^2 - 49)}{1036800}$\\[0.1cm]

$10$ & $\frac{(10m^8 - 2993m^6 + 114597m^4 - 1184767m^2 + 3250433)(m^2 - 1)}{479001600}$ \\[0.1cm]

$11$ &  $\frac{(2m^4 - 49m^2 + 173)(m^2 - 1)(m^2 - 9)(m^2 - 81)}{21288960}$\\[0.1cm]

$12$ & $-\frac{(1382m^{10} - 653818m^8 + 42418211m^6 - 845983589m^4 + 6117468907m^2 - 13695779093)(m^2 - 1)}{2615348736000}$\\ [0.1cm]

\hline
\end{tabular}
\end{table}

\section*{Dedication}
\begin{small}
The author dedicates the work to his spiritual master Bhagawan Sri Sathya Sai Baba. 
\end{small}

% -----------------------------------------------------------
\bibliographystyle{amsplain}

\end{document}